\documentclass[titlepage,12pt]{article} 
\usepackage{hyperref}
\usepackage{amssymb,amsthm,amsmath} 
\usepackage[a4paper]{geometry}
\usepackage{datetime2}
\usepackage[utf8]{inputenc}
\usepackage[italian,english]{babel}
	
\date{}

\newcommand{\ep}{\varepsilon}
\newcommand{\re}{\mathbb{R}}

\newcommand{\n}{\mathbb{N}}
\newcommand{\z}{\mathbb{Z}}

\newcommand{\holder}{H\"older}
\newcommand{\hatU}{\widehat{U}}
\newcommand{\hatf}{f_{*}}
\newcommand{\hatphi}{\widehat{\varphi}}
\newcommand{\meas}{\operatorname{meas}}
\renewcommand{\div}{\operatorname{div}}
\newcommand{\omegalim}{\operatorname*{\omega\,--\,lim}}
\newcommand{\FF}{\mathcal{F}}
\newcommand{\CC}{\mathcal{C}}
\newcommand{\GG}{\mathcal{G}}
\newcommand{\VV}{\mathcal{V}}
\newcommand{\XX}{\mathbb{X}}
\newcommand{\distF}{\operatorname{dist}_{\FF}}
\newcommand{\dist}{\operatorname{dist}}

\newtheorem{thm}{Theorem}[section]
\newtheorem{thmbibl}{Theorem}

\newtheorem{rmk}[thm]{Remark}

\newtheorem{defn}[thm]{Definition}

\newtheorem{lemma}[thm]{Lemma}
\newtheorem{open}{Open problem}

\title{Residual pathologies}

\author{Marina Ghisi\vspace{1ex}\\ 
{\normalsize Universit\`a degli Studi di Pisa} \\
{\normalsize Dipartimento di Matematica}\\ 
{\normalsize PISA (Italy)}\\
{\normalsize e-mail: \texttt{marina.ghisi@unipi.it}}
\and
Massimo Gobbino\vspace{1ex}\\ 
{\normalsize Universit\`a degli Studi di Pisa} \\
{\normalsize Dipartimento di Ingegneria Civile e Industriale}\\ 
{\normalsize PISA (Italy)}\\  
{\normalsize e-mail: \texttt{massimo.gobbino@unipi.it}}
}

\begin{document}
\maketitle

\begin{abstract}

Several counterexamples in analysis show the existence of some special object with some sort of pathological behavior. We present three different examples where the pathological behavior is not an isolated exception, but it is the ``typical'' behavior of the ``generic'' object in a suitable class, where here generic means residual in the sense of Baire category.

The first example is the revisitation of a classical result concerning approximate differentiation. The second example is the derivative loss for solutions to linear wave equations with time-dependent \holder\ continuous propagation speed. The third result is the derivative loss for solutions to transport equations with non-Lipschitz velocity field.

\vspace{6ex}

\noindent{\bf Mathematics Subject Classification 2010 (MSC2010):} 
26A24, 35L15, 35Q35, 76F25.


\vspace{6ex}


\end{abstract}

 
\section{Introduction}

In this paper we consider three different pathologies in analysis (we refer to the following sections for more precise statements and references).

\begin{enumerate}

\item  There exists a nonincreasing function $f:\re\to\re$ of class $C^{1}$, whose derivative $f'$ is \holder\ continuous and approximately differentiable almost everywhere, and such that for every $g:\re\to\re$ of class $C^{2}$ the set $\{x\in\re:f(x)=g(x)\}$ has Lebesgue measure equal to zero.

\item  There exists a \holder\ continuous function $c:[0,+\infty)\to[1,2]$ such that the wave equation $u_{tt}=c(t)u_{xx}$ (for example with Dirichlet boundary conditions in an interval) is not well-posed in $C^{\infty}$, in the sense that there do exist initial data of class $C^{\infty}$ (and actually also in suitable Gevrey spaces) for which the solution (which always exists in a very weak sense) is not even a distribution for all positive times.

\item  There exists a divergence-free vector field $u$ with $W^{1,p}$ regularity (but not Lipschitz continuous) such that the transport equation $\partial_{t}\rho+u\cdot\nabla\rho=0$ admits a solution that is of class $C^{\infty}$ at time $t=0$, but immediately loses all its derivatives for all positive times.

\end{enumerate}

Despite concerning three different topics, the classical constructions of these counterexamples follow a similar strategy that involves two main steps, which we can roughly describe as follows.

\begin{itemize}

\item  In the first step one introduces some sort of basic ingredient. This is in general a smooth object, not a pathological one, but its behavior is border-line in the class of smooth objects, for example in the sense that it saturates some growth or decay estimate. This basic object can be rescaled in order to reproduce the same behavior at different scales that are relevant to the problem.

\item  In the second step one plays with the rescalings of the basic ingredient, and with some sort of iterative procedure (for example a series, a nested construction, a piecewise definition) one produces the required counterexample.

\end{itemize}

Finding the right basic ingredient in the first step usually requires clever ideas. Also devising the right iterative procedure for the second step is sometimes tricky, but after that a lot of dirty work is needed in order to settle the parameters, with cumbersome computations and fine estimates (as in the case of the counterexample for the wave equation).

At the end of the day one has produced \emph{just one} pathological object. Moreover, the very clever choice of the basic ingredient, and the extreme care needed in the iteration process, could lead to suspect that the pathological object is the result of some sort of ``perfect storm'', where everything that potentially could go wrong actually went wrong.

In this paper we show that, at least in the three examples quoted above, the pathological behavior is not uncommon, but on the contrary it is the typical behavior of the generic object. Roughly speaking, instead of being clever one could just pick an object at random, and this random object is very likely to exhibit the pathological behavior!

\paragraph{\textmd{\textit{Baire category theorem}}}

The key ingredient in our approach is Baire category theorem, a tool that has been applied in several areas of analysis and geometry in order to produce different classes of objects, either well behaved or pathological. Any list of examples or references is for sure largely incomplete. The first and perhaps best known example is the construction of continuous functions that are nowhere differentiable. In the same spirit, Baire categorical arguments can be used to prove that residually many functions of class $C^{0,\alpha}$ do not have the Lusin property with functions of class $C^{0,\beta}$ for $\beta>\alpha$, or that residually many Lipschitz functions are not differentiable on a given compact set with measure zero (see~\cite{2001-TAMS-BorMooWan} for more recent applications to differentiability theory).

In a different context, Baire category arguments have been used in the construction of both Besicovich sets~\cite{2003-StMa-Korner}, and of approximate isometries in the spirit of the Nash-Kuiper Theorem (see for example~\cite{2015-CMHe-KirSpaSze}). Moreover, they have been largely exploited in the theory of differential inclusions, starting from the pioneering papers~\cite{1980-Lincei-Cellina,1997-ACTA-DacMar}.

Nevertheless, some part of the PDE community is not quite used to the advantages of Baire category theorem for the constructions of examples, and we hope that this paper could contribute to filling this gap. From the technical point of view, the classical approach to the three examples we discuss has the same structure, consisting on an iteration procedure starting from a basic ingredient. Here we exploit again the same basic ingredients and their scaling properties, but now the dirty work of the iteration process is done in the background by the Baire category theorem, leading in some cases to shorter proofs (as in the case of the wave equation). We hope that getting rid of the technicalities could help in the construction of even more pathological counterexamples that we mention as open problems in each section. 

\paragraph{\textmd{\textit{Overview of the technique -- Solution to a simple exercise}}}

Instead of an abstract description, we present our strategy by applying it to a standard exercise, where we can show in a few lines the main path (functional setting, quantitative formulation, regularization of the center, and rescaled use of some basic ingredient) that we are going to pursue throughout the paper.

Let us prove that there exists a function $f:\re\to\re$ that is \holder\ continuous of order $1/2$ in $\re$, but it is not Lipschitz continuous in any interval $(a,b)\subseteq\re$.

\begin{itemize}

\item \emph{Functional setting.} We fix a real number $H>0$, and we consider the space $\XX$ of all bounded functions $f:\re\to\re$ such that
\begin{equation}
|f(y)-f(x)|\leq H|y-x|^{1/2}
\qquad
\forall(x,y)\in\re^{2}.
\label{hp:basic-holder}
\end{equation}

One can show that $\XX$ is a complete metric space with respect to the usual distance (the $L^{\infty}$ norm of the difference).

\item \emph{Qualitative vs quantitative ``non-pathological'' behavior.} Let $\CC$ denote the set of all $f\in\XX$ that are ``non-counterexamples'', meaning that $f$ is Lipschitz continuous in some interval $(a,b)\subseteq\re$. This is a qualitative property, because we do not know where $f$ is Lipschitz continuous, and with which constant. Let us make it quantitative. For every positive integer $k$, we consider the set $\CC_{k}$ of all $f\in\XX$ for which there exists $z\in[-k,k]$ (depending on $f$) such that
\begin{equation}
|f(y)-f(x)|\leq k|y-x|
\qquad
\forall(x,y)\in[z,z+1/k]^{2}.
\nonumber
\end{equation}

Now we have a more quantitative information on the location and the length of the interval where $f$ is Lipschitz continuous, and on the Lipschitz constant. 

One can show that, for every positive integer $k$, the set $\CC_{k}$ is closed in $\XX$, and the union of all $\CC_{k}$'s is $\CC$.

\item  \emph{Empty interior.} We need to check that $C_{k}$ has empty interior. If this is not the case, then there exist a positive integer $k_{0}$, a function $f_{0}\in\XX$, and a real number $\ep_{0}>0$ such that $B_{\XX}(f_{0},\ep_{0})\subseteq\CC_{k_{0}}$, where $B_{\XX}(f_{0},\ep_{0})$ denotes the ball in $\XX$ with center in $f_{0}$ and radius $\ep_{0}$. Now we show in three steps that this is absurd.

\begin{itemize}

\item  \emph{Basic ingredient.} Let $\varphi:\re\to\re$ be the function with period~1 such that $\varphi(x)=|x|$ for every $x\in[-1/2,1/2]$. This function is globally Lipschitz continuous with constant~1, and globally \holder\ continuous of order $1/2$ with some constant $H_{\varphi}$. The function $\varphi$ is border-line in the sense that it has the same Lipschitz constant in every interval $(a,b)\subseteq\re$. More important, for every positive integer $n$ the rescaled function $n^{-1}\varphi(n^{2}x)$ is Lipschitz continuous with constant $n$ in every interval, and it is globally \holder\ continuous of order 1/2 with the same constant $H_{\varphi}$, independent of $n$.

\item  \emph{Regularization of the center.} Up to modifying $f_{0}$ and restricting $\ep_{0}$, we can always assume that $f_{0}$ is Lipschitz continuous in $\re$ with some constant $L_{0}$ (and also of class $C^{\infty}$, but we do not need it). We can also assume that $f_{0}$ does not saturate inequality (\ref{hp:basic-holder}), in the sense that there exists $\ep_{1}\in(0,1)$ such that
\begin{equation}
|f_{0}(y)-f_{0}(x)|\leq (1-\ep_{1})H|y-x|^{1/2}
\qquad
\forall(x,y)\in\re^{2}.
\nonumber
\end{equation}

\item  \emph{Final contradiction.} Now that $f_{0}$ has been modified as above, we can set
\begin{equation}
f_{n}(x):=f_{0}(x)+\frac{H\ep_{1}}{H_{\varphi}}\cdot\frac{1}{n}\varphi(n^{2}x)
\qquad
\forall x\in\re,
\nonumber
\end{equation}
and check that, for $n$ large enough, $f_{n}\in B_{\XX}(f_{0},\ep_{0})$ (just note that $f_{n}$ is \holder\ continuous of order $1/2$ with constant $H$, and $f_{n}\to f_{0}$ uniformly), but $f_{n}\not\in\CC_{k_{0}}$ (the Lipschitz constant of $f_{n}$ in any fixed interval blows up with $n$). This provides the required contradiction.

\end{itemize}

\end{itemize}

This strategy provides a residual set of counterexamples. It suggests also that an explicit example can be cooked up by considering the function
\begin{equation}
f(x):=\sum_{n=1}^{\infty}\frac{1}{a_{n}}\varphi(b_{n}x)
\qquad
\forall x\in\re,
\nonumber
\end{equation}
where $\left\{a_{n}\right\}$ and $\left\{b_{n}\right\}$ are two sequences of positive real numbers with fast enough growth.

\paragraph{\textmd{\textit{Structure of the paper}}}

In the following sections of this paper we examine the three examples, according to the chronological order of the original papers. The three sections can be read independently. We tried to keep them as self-contained as possible, but of course some familiarity with the original paper is useful in each case. In section~\ref{sec:kohn} we consider the problem concerning approximate differentiation. In section~\ref{sec:dgcs} we consider the derivative loss for the wave equation with \holder\ continuous propagation speed. In section~\ref{sec:alberti} we consider the derivative loss for the transport equation with non-Lipschitz velocity field. 


\setcounter{equation}{0}
\section{Approximate differentiation}\label{sec:kohn}

\subsection{Statement of the problem and previous result}

\paragraph{\textmd{\textit{General setting}}}

For every pair of functions $f:\re\to\re$ and $g:\re\to\re$, let
\begin{equation}
Z(f,g):=\{x\in\re: f(x)=g(x)\}
\label{defn:Z}
\end{equation}
denote the set where they coincide. Let us assume that $f$ is of class $C^{k}$, and that its derivative $f^{(k)}$ of order $k$ is differentiable at almost every $x\in\re$, or at least it satisfies the weaker condition
\begin{equation}
\limsup_{h\to 0}\frac{|f^{(k)}(x+h)-f^{(k)}(x)|}{|h|}<+\infty
\qquad
\mbox{for almost every $x\in\re$.}
\label{hp:limsup}
\end{equation}

Then it is possible to show (see~\cite[3.1.15]{federer:GMT}) that $f$ coincides with a function of class $C^{k+1}$ up to an arbitrarily small set. More precisely, for every $\ep>0$ there exists a function $\varphi_{\ep}:\re\to\re$ of class $C^{k+1}$ such that the set $\re\setminus Z(f,\varphi_{\ep})$ has Lebesgue measure less than $\ep$.

\paragraph{\textmd{\textit{Previous result}}}

In~\cite[3.1.17]{federer:GMT} it was asked whether the same conclusion remains true when differentiability is replaced by approximate differentiability, or the $\limsup$ in (\ref{hp:limsup}) is replaced by the approximate $\limsup$. The answer is positive when $k=0$ (see~\cite[3.1.16]{federer:GMT}), but negative for $k=1$ (and hence also for $k\geq 1$), as shown by the following result.

\begin{thmbibl}[main result in~\cite{kohn}]\label{thmbibl:kohn}

For every real number $\alpha\in(0,1)$ there exists a function $f:\re\to\re$ of class $C^{1}$ satisfying the following four properties:
\begin{enumerate}

\item[\emph{(A1)}] $f'(x)\geq 0$ for every $x\in\re$,

\item[\emph{(A2)}] $f'$ is \holder\ continuous of order $\alpha$ in $\re$,

\item[\emph{(A3)}] $f'$ is approximately differentiable at almost every $x\in\re$,

\item[\emph{(A4)}] for every $g:\re\to\re$ of class $C^{2}$, the coincidence set $Z(f,g)$ defined by (\ref{defn:Z}) has Lebesgue measure equal to~0.

\end{enumerate}

\end{thmbibl}

The previous result was recently quoted (see~\cite[Example~9]{menne:survey}) in order to show that an integral varifold, for which first order quantities such as tangent planes are well defined, in general does not carry any second order information such as curvatures.

\paragraph{\textmd{\textit{The basic ingredient}}}

Let us fix once for all a single-bump function, namely a nonnegative function $\varphi:\re\to[0,+\infty)$ of class $C^{1}$ such that $\varphi(x)=0$ if $x\not\in(-1,1)$, and
\begin{equation}
\int_{-1}^{1}\varphi(x)\,dx=1.
\label{hp:sb-int}
\end{equation}

Then we choose a real number $\beta>1$ (that necessarily exists) such that
\begin{equation}
(\alpha+1)\beta<2,
\label{hp:alpha-beta}
\end{equation}
and for every pair of positive integers $n$ and $k$ we consider the multi-bump function
\begin{equation}
\varphi_{n,k}(x):=\frac{1}{2^{\alpha\beta n}}\sum_{|j|\leq 2^{n}k}\varphi\left(2^{\beta n}\left(x-\frac{j}{2^{n}}\right)\right)
\qquad
\forall x\in\re,
\label{defn:mb}
\end{equation}
and its primitive function
\begin{equation}
\psi_{n,k}(x):=\int_{0}^{x}\varphi_{n,k}(s)\,ds
\qquad
\forall x\in\re.
\label{defn:MB}
\end{equation}

In~\cite{kohn} it was shown that there exist two (rapidly increasing) sequences $\{n_{i}\}$ and $\{a_{i}\}$ of positive integers such that the function
\begin{equation}
\sum_{i=1}^{\infty}\frac{1}{a_{i}}\psi_{n_{i},k}(x)
\nonumber
\end{equation}
provides the required counterexample (the value of $k$ in this case is not relevant, and one could also take $k=+\infty$). In this section we exploit again the basic ingredients (\ref{defn:mb}) and (\ref{defn:MB}), but within our Baire category approach.


\subsection{Choice of the functional space and our result}

The main difficulty in this example consists in giving a structure of complete metric space to a large enough set of approximately differentiable functions. The key point is finding some form of quantitative bound on the approximate differential that is stable with respect to uniform convergence. To this end, we fix first a family of intervals.

\paragraph{\textmd{\textit{A family of intervals}}}

Let $\alpha\in(0,1)$ and $\beta>1$ be two real numbers. For every pair of integers $n$ and $j$, with $n\geq 1$, we consider the interval
\begin{equation}
I_{n,j}:=\left(\frac{j}{2^{n}}-\frac{1}{2^{\beta n}},\frac{j}{2^{n}}+\frac{1}{2^{\beta n}}\right),
\label{def:I-ij}
\end{equation}
and we set
\begin{equation}
U_{n}:=\bigcup_{j\in\z}I_{n,j}.
\label{defn:U-i}
\end{equation}

Since $\beta>1$, when $n$ is large enough the set $U_{n}$ is the disjoint union of countably many open intervals centered in the points of the form $j/2^{n}$. For every positive integer $n$, we consider the open set
\begin{equation}
\hatU_{n}:=\bigcup_{i=n}^{+\infty}U_{i},
\nonumber
\end{equation}
and the complement set
\begin{equation}
K_{n}:=\re\setminus\hatU_{n}.
\nonumber
\end{equation}

It turns out that $\hatU_{n}$ is a dense open set, but the measure of its intersection with any fixed interval tends to~0 as $n\to+\infty$. Indeed, let $[a,b]\subseteq\re$ be a fixed interval. For every positive integer $n$, the interval $[a,b]$ intersects at most $\lfloor(b-a)2^{n}\rfloor+3$ intervals of the form $I_{n,j}$. It follows that
\begin{equation}
\meas\left([a,b]\cap U_{n}\strut\right)\leq \left((b-a)2^{n}+3\strut\right)\frac{2}{2^{\beta n}}
\qquad
\forall n\geq 1,
\nonumber
\end{equation}
where $\meas$ denotes the Lebesgue measure, and as a consequence
\begin{equation}
\meas\left([a,b]\cap\hatU_{n}\right)\leq\sum_{i=n}^{\infty}\left((b-a)2^{i}+3\right)\frac{2}{2^{\beta i}}
\qquad
\forall n\geq 1.
\nonumber
\end{equation}

Since $\beta>1$, the right-hand side tends to 0 as $n\to +\infty$, and therefore 
\begin{equation}
\meas\left([a,b]\setminus\bigcup_{n=1}^{\infty}K_{n}\right)=\lim_{n\to+\infty}\meas\left([a,b]\setminus K_{n}\strut\right)=0.
\label{th:Kn-full}
\end{equation}

\paragraph{\textmd{\textit{Our functional space}}}

Let $C^{1}_{b}(\re)$ denote the set of functions $f:\re\to\re$ such that both $f$ and $f'$ are bounded in $\re$. It is well-known that $C^{1}_{b}(\re)$ is a Banach space with respect to the norm
\begin{equation}
\|f\|_{C^{1}_{b}(\re)}:=\sup_{x\in\re}|f(x)|+\sup_{x\in\re}|f'(x)|.
\nonumber
\end{equation}

Let $\alpha$, $\beta$, $H$ be three real numbers with
\begin{equation}
\alpha\in(0,1),
\qquad\qquad
\beta>1,
\qquad\qquad
H>0.
\label{hp:abc}
\end{equation}

Let $\{K_{n}\}$ be the sequence of closed sets defined above, and let $\{\Lambda_{n}\}$ be an increasing sequence of real numbers such that
\begin{equation}
\Lambda_{n}\geq 2^{(1-\alpha)\beta n}
\qquad
\forall n\in\n.
\label{hp:Ln}
\end{equation}

We consider the space $\XX$ of all functions $f\in C^{1}_{b}(\re)$ whose derivative $f'$ satisfies the following three inequalities:
\begin{gather}
f'(x)\geq 0
\qquad
\forall x\in\re,
\label{defn:F-sign} \\[1ex]
|f'(x)-f'(y)|\leq H|x-y|^{\alpha}
\qquad
\forall (x,y)\in\re^{2},
\label{defn:F-holder}  \\[1ex]
|f'(x)-f'(y)|\leq \Lambda_{n}|x-y|
\qquad
\forall n\geq 1,\ \forall (x,y)\in K_{n}^{2}.
\label{defn:F-lipschitz}
\end{gather}

We point out that in (\ref{defn:F-holder}) the \holder\ continuity of $f'$ is required to be global in $\re$, while in (\ref{defn:F-lipschitz}) the Lipschitz continuity of $f'$ is required only in $K_{n}$, and with a constant that blows up with $n$.

\paragraph{\textmd{\textit{Our result}}}

Since the inequalities in (\ref{defn:F-sign}), (\ref{defn:F-holder}), (\ref{defn:F-lipschitz}) are stable under uniform convergence, the set $\XX$ is closed in $C^{1}_{b}(\re)$, and therefore it is a complete metric space with respect to the distance inherited from the ambient space. Moreover, from (\ref{defn:F-sign}) and (\ref{defn:F-holder}) it follows that every $f\in\XX$ satisfies (A1) and (A2) of Theorem~\ref{thmbibl:kohn}. 

Our result is that all the elements of $\XX$ satisfy also (A3), and that a residual subset satisfies (A4), and also the following (slightly stronger) condition
\begin{enumerate}

\item[(A4-s)]  for every $g:\re\to\re$ of class $C^{1,1}$ (namely of class $C^{1}$ with $g'$ Lipschitz continuous), the coincidence set $Z(f,g)$ has Lebesgue measure equal to~0.

\end{enumerate}

\begin{thm}\label{thm:kohn-residual}

Let $\alpha$, $\beta$, $H$ be real numbers satisfying (\ref{hp:abc}). Let us define as above the sequence of closed sets $\{K_{n}\}$, the sequence or real numbers $\{\Lambda_{n}\}$, and the space $\XX$.

Then the following statements hold true.
\begin{itemize}						

\item  Every $f\in\XX$ satisfies condition (A3) of Theorem~\ref{thmbibl:kohn}.

\item  If in addition (\ref{hp:alpha-beta}) holds true, then the set of all functions $f\in\XX$ that satisfy condition (A4-s) above is residual in~$\XX$.

\end{itemize}

\end{thm}

We conclude by mentioning a possible extension and an open problem.

\begin{rmk}
\begin{em}

It should not be difficult to strengthen the results of Theorem~\ref{thmbibl:kohn} and Theorem~\ref{thm:kohn-residual} by asking that $f'$ is \holder\ continuous of \emph{any} order $\alpha\in(0,1)$, and not just with a fixed $\alpha$, or by asking that $f'$ has a given non-Lipschitz continuity modulus, for example of log-Lipschitz type. Such an extension probably requires only a new definition of the intervals $I_{n,j}$, with a radius that decays faster.

\end{em}
\end{rmk}

\begin{open}
\begin{em}

Is it possible to find a counterexample as in Theorem~\ref{thmbibl:kohn}, with the further requirement that $f'$ is approximately differentiable at every $x\in\re$ (and not just almost every)?

\end{em}
\end{open}

We point out that the stronger form stated in the open problem is actually the original question posed in~\cite[3.1.15]{federer:GMT}.


\subsection{Technical preliminaries}

In this subsection we prove some technical results that we need in the proof of Theorem~\ref{thm:kohn-residual}. The first one concerns the $\omega$-limit of a sequence of sets. Let us recall the definition.
Let $\{M_{n}\}$ be a sequence of subsets of $\re$. The $\omega$-limit of the sequence is the set $M_{\infty}$ defined as
\begin{equation}
M_{\infty}:=\bigcap_{n\geq 0}\:\overline{\bigcup_{i\geq n}M_{i}}.
\nonumber
\end{equation}

It is well-know that $M_{\infty}$ can be characterized as the set of points $x\in\re$ for which there exist an increasing sequence $n_{k}$ of positive integers, and a sequence $x_{k}\to x$ such that $x_{k}\in M_{n_{k}}$ for every positive integer $k$.

In particular, if $f_{n}:\re\to\re$ and  $g_{n}:\re\to\re$ are continuous functions, and if $f_{n}\to f_{\infty}$ and $g_{n}\to g_{\infty}$ uniformly in $\re$, then one can check that
\begin{equation}
Z(f_{\infty},g_{\infty})\supseteq\omegalim_{n\to +\infty}Z(f_{n},g_{n}).
\nonumber
\end{equation}

If all the sets $M_{n}$ are measurable and contained in a common compact set, then their $\omega$-limit is not smaller, in the measure theoretic sense, than the elements of the sequence, as shown in the following result.

\begin{lemma}[Measure of the $\omega$-limit set]\label{lemma:omega-lim}

Let $\{M_{n}\}$ be a sequence of measurable subsets of $\re$ contained in a common interval $[a,b]$.

Then it turns out that
\begin{equation}
\meas\left(\omegalim_{n\to +\infty}M_{n}\right)\geq\limsup_{n\to+\infty}\meas(M_{n}).
\nonumber
\end{equation} 

\end{lemma}

\begin{proof}

Let us set
\begin{equation}
A_{n}:=\overline{\bigcup_{i\geq n}M_{i}}.
\qquad
\forall n\in\n.
\nonumber
\end{equation}

It turns out that $\{A_{n}\}$ is a decreasing sequence of measurable sets with finite measure. Since $A_{n}\supseteq M_{n}$ for every $n\in\n$, and since the $\omega$-limit is the intersection of all $A_{n}$'s, it follows that
\begin{equation}
\meas\left(\omegalim_{n\to +\infty}M_{n}\right)=\lim_{n\to +\infty}\meas(A_{n})\geq\limsup_{n\to+\infty}\meas(M_{n}),
\nonumber
\end{equation}
which completes the proof.
\end{proof}


The second result contains the key properties of the multi-bump functions. Similar estimates are crucial also in~\cite{kohn}.

\begin{lemma}[Properties of the multi-bump function]\label{lemma:multi-bump}

For every pair of positive integers $n$ and $k$, let us consider the multi-bump function $\varphi_{n,k}$ defined in (\ref{defn:mb}), and its antiderivative $\psi_{n,k}$ defined in (\ref{defn:MB}). Let us assume that $n$ is sufficiently large so that
\begin{equation}
\beta n\geq n+2.
\label{hp:mb-large}
\end{equation}

Then the following estimates hold true.

\begin{itemize}

\item (Support) The function $\varphi_{n,k}(x)$ vanishes when $|x|\geq k+1$ and outside the set $U_{n}$ defined in (\ref{defn:U-i}), and in particular
\begin{equation}
\varphi_{n,k}(x)=0
\qquad
\forall x\in K_{n}.
\label{th:mb-support}
\end{equation}

\item  (Pointwise estimates on $\varphi_{n,k}$ and $\psi_{n,k}$) There exists a constant $M_{\varphi}$ such that
\begin{equation}
|\varphi_{n,k}(x)|\leq\frac{M_{\varphi}}{2^{\alpha\beta n}}
\qquad
\forall x\in\re
\label{th:mb-phi-infty}
\end{equation}
and
\begin{equation}
|\psi_{n,k}(x)|\leq\frac{(k+1)M_{\varphi}}{2^{\alpha\beta n}}
\qquad
\forall x\in\re.
\label{th:mb-psi-infty}
\end{equation}

\item  (Lipschitz constant) There exists a constant $L_{\varphi}$ such that
\begin{equation}
|\varphi_{n,k}(y)-\varphi_{n,k}(x)|\leq 2^{(1-\alpha)\beta n}L_{\varphi}|y-x|
\qquad
\forall(x,y)\in\re^{2}.
\label{th:mb-lip}
\end{equation}

\item  (\holder\ constant)  There exists a constant $H_{\varphi}$ such that
\begin{equation}
|\varphi_{n,k}(y)-\varphi_{n,k}(x)|\leq H_{\varphi}|y-x|^{\alpha}
\qquad
\forall(x,y)\in\re^{2}.
\label{th:mb-holder}
\end{equation}

\item  (Gap estimate)
If $x$ and $y$ are real numbers in $[-k,k]$ such that $y-x\geq 3/2^{n}$, then
\begin{equation}
\psi_{n,k}(y)-\psi_{n,k}(x)\geq\frac{1}{2^{(\alpha+1)\beta n}}.
\label{th:mb-gap}
\end{equation}

\end{itemize}

\end{lemma}

\begin{proof}

To begin with, we observe that, under condition (\ref{hp:mb-large}), the multi-bump function $\varphi_{n,k}$ is the sum of $2\cdot 2^{n}k+1$ single-bump functions with disjoint supports. The support of each term is contained in the interval $I_{n,j}$ as defined in (\ref{def:I-ij}), and this implies (\ref{th:mb-support}).

As for the pointwise and the Lipschitz estimates, let $M_{\varphi}$ denote the maximum of $\varphi$, and let $L_{\varphi}$ denote the maximum of $\varphi'$, which is also the Lipschitz constant of $\varphi$. At this point (\ref{th:mb-phi-infty}) and (\ref{th:mb-lip}) follow from the disjoint supports, while (\ref{th:mb-psi-infty}) follows from (\ref{th:mb-phi-infty}) by integration (we recall that $\varphi_{n,k}(x)=0$ when $|x|\geq k+1$). 

As for the \holder\ estimates, let $H_{\varphi}$ denote the $\alpha$-\holder\ constant of $\varphi$ (it is finite because $\varphi$ has compact support). By a scaling argument, the $\alpha$-\holder\ constant of each term in the sum (\ref{defn:mb}) is again $H_{\varphi}$. We claim that also the $\alpha$-\holder\ constant of the sum is $H_{\varphi}$. Indeed, due to the multi-bump structure of $\varphi_{n,k}$, for every pair $(x,y)\in\re^{2}$ there exist points $x_{*}$ and $y_{*}$ in $I_{n,0}$ (namely ``in the same bump'') such that
\begin{equation}
\varphi_{n,k}(x)=\varphi_{n,k}(x_{*}),
\qquad
\varphi_{n,k}(y)=\varphi_{n,k}(y_{*}),
\qquad
|y_{*}-x_{*}|\leq|y-x|.
\nonumber
\end{equation}

The existence of $x_{*}$ and $y_{*}$ follows from condition (\ref{hp:mb-large}), which guarantees that any two points in different bumps are more distant than any two points in the same bump.

Finally, let us prove the gap estimate. From the two conditions on $x$ and $y$ we deduce that there exists $j_{0}\in\z$ such that $|j_{0}|\leq 2^{n}k$ and
\begin{equation}
x\leq\frac{j_{0}-1}{2^{n}}<\frac{j_{0}+1}{2^{n}}\leq y,
\nonumber
\end{equation}
and in particular $I_{n,j_{0}}\subseteq[x,y]$. Since $\varphi_{n,k}$ is nonnegative, it follows that
\begin{equation}
\psi_{n,k}(y)-\psi_{n,k}(x)=\int_{x}^{y}\varphi_{n,k}(s)\,ds\geq\int_{I_{n,j_{0}}}\varphi_{n,k}(s)\,ds.
\nonumber
\end{equation}

The last integral involves a single bump, and can be computed with a variable change. Recalling (\ref{hp:sb-int}), we obtain (\ref{th:mb-gap}).
\end{proof}


The third result shows that one can extend a function from a closed set to its convex hull without modifying its \holder\ or Lipschitz constants. We state it in a slightly different way that is more suited for the applications to our space $\XX$.

\begin{lemma}[Extension]\label{lemma:extension}

Let $\varphi:\re\to[0,+\infty)$ be a nonnegative function. Let  us assume that there exists a closed set $K\subseteq\re$, and three real numbers $\alpha\in(0,1)$, $H\geq 0$, and $L\geq 0$ such that
\begin{gather}
|\varphi(x)-\varphi(y)|\leq H|x-y|^{\alpha}
\qquad
\forall(x,y)\in\re^{2},
\label{hp:approx-holder}  \\[1ex]
|\varphi(x)-\varphi(y)|\leq L|x-y|
\qquad
\forall(x,y)\in K^{2}.
\nonumber
\end{gather}

Let $C$ and $D$ be elements of $K$, with $C<D$.

Then there exists a nonnegative function $\hatphi:\re\to[0,+\infty)$ such that
\begin{gather}
\hatphi(x)=\varphi(x)
\qquad
\forall x\in(-\infty,C]\cup K\cup[D,+\infty),  
\label{th:approx=}   \\[1ex]
|\hatphi(x)-\hatphi(y)|\leq H|x-y|^{\alpha}
\qquad
\forall(x,y)\in\re^{2},
\label{th:approx-holder}  \\[1ex]
|\hatphi(x)-\hatphi(y)|\leq L|x-y|
\qquad
\forall(x,y)\in[C,D]^{2},
\label{th:approx-lip}
\end{gather}
and in addition
\begin{equation}
\left|\varphi(x)-\hatphi(x)\right|\leq 2H\meas\left([C,D]\setminus K\strut\right)^{\alpha}
\qquad
\forall x\in\re.	
\label{th:approx-est}
\end{equation}

\end{lemma}

\begin{proof}

The idea is to define $\hatphi$ by extending $\varphi$ from $[C,D]\cap K$ to $[C,D]$ in a piecewise affine way. The main point is showing that this extension does not change the $\alpha$-\holder\ constant or the Lipschitz constant.

To begin with, we observe that $[C,D]\setminus K$ is an open set, and therefore all its connected components are intervals. Let $\{(a_{i},b_{i})\}_{i\in I}$ denote the set of these connected components, where the index set $I$ is finite or countable. At this point we can set
\begin{equation}
\hatphi(x):=\left\{
\begin{array}{ll}
\varphi(x) & \mbox{if }x\in(-\infty,C]\cup K\cup[D,+\infty), \\[1ex]
\varphi(a_{i})+\dfrac{\varphi(b_{i})-\varphi(a_{i})}{b_{i}-a_{i}}(x-a_{i}) & \mbox{if $x\in(a_{i},b_{i})$ for some $i\in I$.}
\end{array}\right.
\label{defn:hatphi}
\end{equation}

We observe that every $x\in(a_{i},b_{i})$ can be written in the form $x=\lambda a_{i}+(1-\lambda)b_{i}$ for some $\lambda\in(0,1)$, and in this case it turns out that $\hatphi(x)=\lambda \varphi(a_{i})+(1-\lambda)\varphi(b_{i})$. This shows that $\hatphi$ is nonnegative if $\varphi$ is nonnegative. From the definition, if follows also that $\varphi$ and $\hatphi$ coincide in $K$ and outside $[C,D]$.

Let us show that $\hatphi$ has the same \holder\ constant of $\varphi$ (the argument for the Lipschitz constant is analogous, just with exponent~1 instead of $\alpha$, constant $L$ instead of $H$, and $x$ and $y$ in $[C,D]$ instead of $\re$). Let us consider real numbers $x$ and $y$, and let us distinguish some cases according to the position of $x$ and $y$ with respect to $K$.
\begin{itemize}

\item If $x$ and $y$ are both in $K$, then (\ref{th:approx-holder}) follows from (\ref{hp:approx-holder}) because $\varphi$ and $\hatphi$ coincide.

\item If $x$ and $y$ are both in the same interval $(a_{i},b_{i})$, then from (\ref{defn:hatphi}) we obtain that
\begin{equation}
|\hatphi(x)-\hatphi(y)|=\frac{|\varphi(b_{i})-\varphi(a_{i})|}{b_{i}-a_{i}}|x-y|\leq H\frac{|x-y|}{(b_{i}-a_{i})^{1-\alpha}}\leq H|x-y|^{\alpha},
\nonumber
\end{equation}
where the first inequality follows from (\ref{hp:approx-holder}), and the second one from the fact that $|x-y|<b_{i}-a_{i}$.

\item  If $x\in K$ and $y\in(a_{i},b_{i})$ for some $i\in I$, then from (\ref{hp:approx-holder}) we obtain that
\begin{equation}
\varphi(a_{i})\leq\varphi(x)+H|a_{i}-x|^{\alpha}=\hatphi(x)+H|a_{i}-x|^{\alpha}
\nonumber
\end{equation}
and
\begin{equation}
\varphi(b_{i})\leq\varphi(x)+H|b_{i}-x|^{\alpha}=\hatphi(x)+H|b_{i}-x|^{\alpha}.
\nonumber
\end{equation}
If $y=\lambda a_{i}+(1-\lambda)b_{i}$ for some $\lambda\in(0,1)$, we deduce that
\begin{eqnarray}
\hatphi(y) & = & \lambda\varphi(a_{i})+(1-\lambda)\varphi(b_{i})  
\nonumber  \\[0.5ex]
& \leq & \hatphi(x)+H\left(\lambda|a_{i}-x|^{\alpha}+(1-\lambda)|b_{i}-x|^{\alpha}\strut\right)  
\nonumber  \\[0.5ex]
& \leq & \hatphi(x)+H|y-x|^{\alpha},
\nonumber
\end{eqnarray}
where the last inequality follows from the concavity of the function $z\to|z-x|^{\alpha}$ in the half-lines $z\leq x$ and $z\geq x$ (here we need that $a_{i}$ and $b_{i}$ lie on the same side with respect to $x$). In an analogous way we can show that
$$\hatphi(y)\geq\hatphi(x)-H|x-y|^{\alpha},$$
and this completes the proof of (\ref{th:approx-holder}) in this case. In a symmetric way we can deal with the case where $y\in K$ and $x\in(a_{i},b_{i})$.

\item  It remains to consider the case where $x\in(a_{j},b_{j})$ and $y\in(a_{i},b_{i})$ for some indices $i\neq j$. Since $a_{j}\in K$ and $b_{j}\in K$, from the result of the previous step we obtain the inequalities
\begin{equation}
\hatphi(a_{j})\leq\hatphi(y)+H|a_{j}-y|^{\alpha}
\qquad\mbox{and}\qquad
\hatphi(b_{j})\leq\hatphi(y)+H|b_{j}-y|^{\alpha}.
\nonumber
\end{equation}

Then we write $x$ in the form $\mu a_{j}+(1-\mu)b_{j}$ for some $\mu\in(0,1)$, and from the previous inequalities we deduce that
\begin{eqnarray}
\hatphi(x) & = & \mu\hatphi(a_{j})+(1-\mu)\hatphi(b_{j})  \nonumber \\[0.5ex]
& \leq & \hatphi(y)+H\left(\mu|a_{j}-y|^{\alpha}+(1-\mu)|b_{j}-y|^{\alpha}\strut\right)  \nonumber \\[0.5ex]
& \leq & \hatphi(y)+H|x-y|^{\alpha},
\nonumber
\end{eqnarray}
where again in the last step we exploited the concavity of the function $z\to|z-y|^{\alpha}$, and the fact that $a_{j}$ and $b_{j}$ lie on the same side with respect to $y$. In an analogous way we can show that $\hatphi(x)\geq\hatphi(y)-H|x-y|^{\alpha}$, and this completes the proof of (\ref{th:approx-holder}) also in the last case.

\end{itemize}

It remains to prove (\ref{th:approx-est}). To this end, we observe that the left-hand side is different from zero only when $x\in(a_{i},b_{i})$ for some $i\in I$. Since $\varphi(a_{i})=\hatphi(a_{i})$, from (\ref{hp:approx-holder}) and (\ref{th:approx-holder}) it follows that
\begin{equation}
|\hatphi(x)-\varphi(x)|\leq|\hatphi(x)-\hatphi(a_{i})|+|\varphi(a_{i})-\varphi(x)|\leq 2H|x-a_{i}|^{\alpha},
\nonumber
\end{equation}
and we conclude by observing that
\begin{equation}
x-a_{i}<b_{i}-a_{i}\leq\meas\left([C,D]\setminus K\strut\right).
\nonumber
\end{equation}

This completes the proof.
\end{proof}


In the last result we show the density in $\XX$ of functions that are of class $C^{1,1}$ in a given interval. This density comes into play in the approximation step of the proof of our main result. Since the definition of $\XX$ involves a different control of the Lipschitz constant in each set $K_{n}$, and these sets are not so nice, it is not clear whether standard approximation tools, such as convolution, can be applied. Our proof relies on~Lemma~\ref{lemma:extension}.

\begin{lemma}[Approximation]\label{lemma:approx}

Let $f\in\XX$ be a function, let $\ep>0$ be a real number, and let $b>a$ be two real numbers.

Then there exists $\hatf\in\XX$ with $\hatf'$ Lipschitz continuous in $[a,b]$, and $\dist_{\XX}(f,\hatf)\leq\ep$.

\end{lemma}

\begin{proof}

Thanks to (\ref{th:Kn-full}), there exists a positive integer $n_{0}$ such that
\begin{equation}
2H\meas\left([a-1,b+1]\setminus K_{n_{0}}\strut\right)^{\alpha}\leq\frac{\ep}{2(b-a+2)}.
\nonumber
\end{equation}

Let us choose two points $a_{0}\in[a-1,a]\cap K_{n_{0}}$ and $b_{0}\in[b,b+1]\cap K_{n_{0}}$, and let us apply Lemma~\ref{lemma:extension} with 
$$\varphi:=f',
\qquad
K:=K_{n_{0}}, 
\qquad
L:=\Lambda_{n_{0}}, 
\qquad
[C,D]:=[a_{0},b_{0}].$$ 

We obtain a function $\hatphi$ satisfying (\ref{th:approx=}) through (\ref{th:approx-est}). We claim that the primitive function
\begin{equation}
\hatf(x):=f(0)+\int_{0}^{x}\hatphi(s)\,ds
\qquad
\forall x\in\re
\nonumber
\end{equation}
has the required properties.

First of all, $\hatf'=\hatphi$ is Lipschitz continuous in $[a,b]$ with constant $\Lambda_{n_{0}}$. 

Let us prove that $\hatf\in\XX$. To this end,  we observe that $\hatf$ satisfies (\ref{defn:F-sign}) because $\hatphi$ is nonnegative, and it satisfies (\ref{defn:F-holder}) because of (\ref{th:approx-holder}). It remains to check (\ref{defn:F-lipschitz}). Let $n$ be a positive integer, and let $x$ and $y$ be in $K_{n}$. If $n\leq n_{0}$, then $K_{n}\subseteq K_{n_{0}}$, and therefore 
\begin{equation}
\hatf'(x)=\hatphi(x)=f'(x)
\qquad\mbox{and}\qquad
\hatf'(y)=\hatphi(y)=f'(y),
\label{eqn:all=}
\end{equation}
so that, for this value of $n$, inequality (\ref{defn:F-lipschitz}) for $\hatf$ follows from the analogous inequality for $f$. If $n\geq n_{0}$ we assume, without loss of generality, that $x<y$, and we distinguish some cases according to the position of $x$ and $y$. 
\begin{itemize}

\item  If both $x$ and $y$ are outside the interval $[a_{0},b_{0}]$, then again (\ref{eqn:all=}) holds true, and therefore there is nothing to prove.

\item  If both $x$ and $y$ lie in the interval $[a_{0},b_{0}]$, then from (\ref{th:approx-lip}) (that holds true with $L=\Lambda_{n_{0}}$, namely the Lipschitz constant of $f'$ in $K_{n_{0}}$) we deduce that
\begin{equation}
\left|\hatf'(x)-\hatf'(y)\right|=\left|\hatphi(x)-\hatphi(y)\right|\leq \Lambda_{n_{0}}|x-y|,
\nonumber
\end{equation}
and we conclude by observing that $\Lambda_{n_{0}}\leq \Lambda_{n}$.

\item  If $x<a_{0}$ and $y\in[a_{0},b_{0}]$, then we observe that $a_{0}\in K_{n_{0}}\subseteq K_{n}$, and therefore
\begin{eqnarray}
\left|\hatf'(x)-\hatf'(y)\right| & \leq & \left|\hatphi(x)-\hatphi(a_{0})\right|+\left|\hatphi(a_{0})-\hatphi(y)\right| 
\nonumber   \\[0.5ex]
& \leq & \Lambda_{n}(a_{0}-x)+\Lambda_{n_{0}}(y-a_{0}) 
\nonumber  \\[0.5ex]
& \leq & \Lambda_{n}(y-x).
\nonumber
\end{eqnarray}
A symmetric argument works if $x$ lies in the interval $[a_{0},b_{0}]$ and $y>b_{0}$.

\end{itemize}

Finally, from (\ref{th:approx-est}) we deduce that
\begin{equation}
\left|f'(x)-\hatf'(x)\right|=\left|f'(x)-\hatphi(x)\right|\leq 2H\meas\left([a_{0},b_{0}]\setminus K_{n_{0}}\strut\right)^{\alpha}\leq\frac{\ep}{2}
\nonumber
\end{equation}
for every $x\in\re$, and since $f'$ and $\hatphi$ coincide outside $[a-1,b+1]$, we deduce also that
\begin{equation}
\left|f(x)-\hatf(x)\right|\leq\int_{a-1}^{b+1}\left|f'(s)-\hatphi(s)\right|\,ds\leq\int_{a-1}^{b+1}\frac{\ep}{2(b-a+2)}\,ds\leq\frac{\ep}{2}
\nonumber
\end{equation}
for every $x\in\re$. Adding these two inequalities we obtain that $\dist_{\XX}(f,\hatf)\leq\ep$, and this completes the proof.
\end{proof}


\subsection{Proof of Theorem~\ref{thm:kohn-residual}}

\subsubsection{Approximate differentiability}

The main idea is the following. Let $\varphi_{1}:\re\to\re$ and $\varphi_{2}:\re\to\re$ be two functions. If $x$ is a point of density~1 for the set $Z(\varphi_{1},\varphi_{2})$, and $\varphi_{2}$ is differentiable at $x$, then $\varphi_{1}$ is approximately differentiable at $x$, and its approximate differential is $\varphi_{2}'(x)$. Our plan is to apply this idea with $\varphi_{1}:=f'$, and $\varphi_{2}$ equal to the Lipschitz extension of $f'$ outside a suitable $K_{n}$ provided by Lemma~\ref{lemma:extension}.

Let $[a,b]$ be a fixed interval. For every positive integer $n$, let us choose real numbers $a_{n}\in[a-1,a]\cap K_{n}$ and $b_{n}\in[b,b+1]\cap K_{n}$, and let us apply Lemma~\ref{lemma:extension} with 
$$\varphi:=f',
\qquad
K:=K_{n},
\qquad
L:=\Lambda_{n},
\qquad
[C,D]:=[a_{n},b_{n}].$$ 

We obtain a Lipschitz function $\widehat{\varphi}_{n}$ that coincides with $f'$ in $[a,b]\cap K_{n}$. Let $E_{n}$ denote the set of points in $[a,b]$ where $\widehat{\varphi}_{n}$ is not differentiable, so that $\meas(E_{n})=0$. Let $K_{n}^{(1)}$ denote the set of points in $[a,b]$ with density~1 with respect to $K_{n}$, and let us set
\begin{equation}
E_{\infty}:=[a,b]\setminus\bigcup_{n=1}^{\infty}K_{n}^{(1)}.
\nonumber
\end{equation} 

Since $\meas(K_{n})=\meas(K_{n}^{(1)})$, from (\ref{th:Kn-full}) it follows that $\meas(E_{\infty})=0$, and therefore also the set
\begin{equation}
E:=E_{\infty}\cup\bigcup_{n=1}^{\infty}E_{n}
\nonumber
\end{equation}
has Lebesgue measure equal to~0.

We claim that $f'$ is approximately differentiable at every point $x_{0}\in[a,b]\setminus E$. Indeed, for any such point $x_{0}$ there exists an integer $n_{0}\geq 1$ such that $x_{0}\in K_{n_{0}}^{(1)}$. Since $f'$ coincides with $\widehat{\varphi}_{n_{0}}$ in $K_{n_{0}}^{(1)}$, and $\widehat{\varphi}_{n_{0}}$ is differentiable in $x_{0}$, it follows that the approximate differential of $f'$ in $x_{0}$ exists. 

Since the interval $[a,b]$ is arbitrary, we have proved that $f'$ is approximately differentiable almost everywhere.


\subsubsection{Residual set}

\paragraph{\textmd{\textit{Quantitative non-pathological behavior}}}

Let $\CC$ denote the set of ``non-counterexamples'', namely the set of all functions $f\in\XX$ for which there exists $g:\re\to\re$ of class $C^{1,1}$ such that the coincidence set $Z(f,g)$ has positive Lebesgue measure. We have to show that $\CC$ is a countable union of closed subsets of $\XX$ with empty interior. 

To this end, we state in a quantitative way the definition of $\CC$. For every positive integer $k$, we consider the set $\CC_{k}$ of all functions $f\in\XX$ for which there exists $g:\re\to\re$ of class $C^{1,1}$ with the following properties:
\begin{gather}
|g(x)|+|g'(x)|\leq k
\qquad
\forall x\in\re,
\label{defn:g1}   \\[1ex]
|g'(x)-g'(y)|\leq k|x-y|
\qquad
\forall (x,y)\in\re^{2},
\label{defn:g2}   \\[1ex]
\exists z\in[-k,k]\quad\forall r\in\left(0,\frac{1}{k}\right)
\quad
\meas\left(Z(f,g)\cap[z-r,z+r]\strut\right)\geq\frac{19}{10}r.
\label{defn:g3}
\end{gather}

In words, (\ref{defn:g1}) is a quantitative boundedness of $g$ and $g'$, (\ref{defn:g2}) is a quantitative Lipschitz continuity of $g'$, and (\ref{defn:g3}) is a sort of quantitative estimate for the position and size of the coincidence set.

The proof is complete if we show that $\CC$ is the union of all $\CC_{k}$'s, and that $\CC_{k}$ is a closed set with empty interior for every positive integer $k$.

\paragraph{\textmd{\textit{The set $\CC$ is the union of all $\CC_{k}$'s}}}

Let $f\in\CC$, and let $g:\re\to\re$ be a function of class $C^{1,1}$ such that $Z(f,g)$ has positive Lebesgue measure. Up to modifying $g$ outside a large enough ball, we can assume that also $g\in C^{1}_{b}(\re)$. At this point $g$ satisfies (\ref{defn:g1}) and (\ref{defn:g2}) provided that $k$ is large enough. Moreover, if we consider any point $z$ of density~1 for the coincidence set $Z(f,g)$, then also (\ref{defn:g3}) is satisfied provided that $k$ is large enough.

This proves that every $f\in\CC$ belongs to $\CC_{k}$ if $k$ is sufficiently large.

\paragraph{\textmd{\textit{The set $\CC_{k}$ is closed}}}

Let $k$ be a \emph{fixed} positive integer. Let $\{f_{n}\}\subseteq\CC_{k}$ be any sequence, and let us assume that $f_{n}\to f_{\infty}$ and $f_{n}'\to f_{\infty}'$ uniformly in $\re$. Let $\{g_{n}\}\subseteq C^{1,1}(\re)$ and $\{z_{n}\}\subseteq[-k,k]$ be the corresponding sequences of functions and points as in the definition of $\CC_{k}$. Due to the uniform bounds on $g_{n}$ and $z_{n}$ provided by (\ref{defn:g1}) through (\ref{defn:g3}), we deduce that (up to subsequences, not relabelled)
\begin{equation}
g_{n}\to g_{\infty}
\quad\mbox{in }C^{1}_{b}(\re)
\qquad\qquad\mbox{and}\qquad\qquad
z_{n}\to z_{\infty}.
\nonumber
\end{equation}

Moreover, the inequalities pass to the limit, and hence $g_{\infty}$ satisfied (\ref{defn:g1}) and (\ref{defn:g2}), and $z_{\infty}$ satisfies $|z_{\infty}|\leq k$.
It remains to show that
\begin{equation}
\meas\left(Z(f_{\infty},g_{\infty})\cap[z_{\infty}-r,z_{\infty}+r]\strut\right)\geq\frac{19}{10}r
\qquad
\forall r\in\left(0,\frac{1}{k}\right).
\nonumber
\end{equation}

To this end, we observe that for every positive real number $r$ it turns out that
\begin{equation}
Z(f_{\infty},g_{\infty})\cap[z_{\infty}-r,z_{\infty}+r]\supseteq\omegalim_{n\to+\infty}Z(f_{n},g_{n})\cap[z_{n}-r,z_{n}+r],
\nonumber
\end{equation}
so that the conclusion follows from Lemma~\ref{lemma:omega-lim}. 

\paragraph{\textmd{\textit{The set $\CC_{k}$ has empty interior}}}

Let us assume by contradiction that there exists an integer $k_{0}\geq 1$, a function $f_{0}\in\XX$, and a real number $\ep_{0}>0$ such that $B_{\XX}(f_{0},\ep_{0})\subseteq\CC_{k_{0}}$, where $B_{\XX}(f_{0},\ep_{0})$ denotes the open ball in $\XX$ with center in $f_{0}$ and radius $\ep_{0}$. 

\subparagraph{\textmd{\textit{Regularization of the center}}}

Due to the approximation Lemma~\ref{lemma:approx}, we can assume that $f_{0}'$ is Lipschitz continuous in $[-k_{0},k_{0}]$ with some constant $L_{0}$. Up to multiplying $f_{0}$ by $1-\ep$ for a small enough $\ep$, we can also assume that $f_{0}'$ does not saturate the inequalities in (\ref{defn:F-holder}) and (\ref{defn:F-lipschitz}), namely that there exists $\ep_{1}\in(0,1)$ such that
\begin{equation}
|f_{0}'(x)-f_{0}'(y)|\leq (1-\ep_{1})H|x-y|^{\alpha}
\qquad
\forall (x,y)\in\re^{2}
\label{est:sat-holder}
\end{equation}
and
\begin{equation}
|f_{0}'(x)-f_{0}'(y)|\leq (1-\ep_{1})\Lambda_{n}|x-y|
\qquad
\forall n\geq 1,\ \forall (x,y)\in K_{n}^{2}.
\label{est:sat-lip}
\end{equation}

\subparagraph{\textmd{\textit{Use of rescaled basic ingredient}}}

Let us choose a real number $\ep_{2}>0$ such that
\begin{equation}
\ep_{2}H_{\varphi}\leq\ep_{1}H,
\qquad
\ep_{2}L_{\varphi}\leq\ep_{1},
\qquad
\ep_{2}(k_{0}+2)M_{\varphi}<\ep_{0},
\label{defn:ep}
\end{equation}
where $H_{\varphi}$, $L_{\varphi}$ and $M_{\varphi}$ are the constants of Lemma~\ref{lemma:multi-bump}, and then let us choose an integer $n_{0}$ large enough such that (note that we can fulfill the last condition because of assumption (\ref{hp:alpha-beta}))
\begin{equation}
\frac{10}{2^{n_{0}}}<\frac{1}{k_{0}},
\qquad\quad
\frac{1}{2^{n_{0}}}>\frac{6}{2^{\beta n_{0}}},
\qquad\quad
\frac{\ep_{2}}{2^{(\alpha+1)\beta n_{0}}}>(k_{0}+L_{0})\frac{400}{2^{2n_{0}}}.
\label{defn:n}
\end{equation}

Let us consider the multi-bump functions $\varphi_{n,k}$ defined in (\ref{defn:mb}) and their antiderivatives $\psi_{n,k}$ defined in (\ref{defn:MB}), and let us set
\begin{equation}
\psi(x):=f_{0}(x)+\ep_{2}\psi_{n_{0},k_{0}}(x)
\qquad
\forall x\in\re.
\label{defn:kohn-psi}
\end{equation}

In the next paragraphs we show that $\psi\in B_{\XX}(f_{0},\ep_{0})$ and $\psi\not\in\CC_{k_{0}}$, which provides the required contradiction.

\subparagraph{\textmd{\textit{Final contradiction -- Proof that $\psi\in B_{\XX}(f_{0},\ep_{0})$}}}

Let us check first that $\psi\in\XX$. Inequality (\ref{defn:F-sign}) is trivial because $f_{0}'$ and $\varphi_{n_{0},k_{0}}$ are nonnegative. Inequality (\ref{defn:F-holder}) follows from (\ref{est:sat-holder}), (\ref{th:mb-holder}), and the first condition in (\ref{defn:ep}). It remains to check (\ref{defn:F-lipschitz}). To this end, let $n$ be a positive integer, and let $x$ and $y$ be in $K_{n}$.

If $n\leq n_{0}$, then $K_{n}\subseteq K_{n_{0}}$, and therefore from (\ref{th:mb-support}) we deduce that $\psi_{n_{0},k_{0}}$ vanishes in $K_{n}$. Thus from (\ref{est:sat-lip}) it follows that
\begin{equation}
|\psi'(y)-\psi'(x)|=|f_{0}'(y)-f_{0}'(x)|\leq (1-\ep_{1})\Lambda_{n}|x-y|\leq \Lambda_{n}|x-y|.
\nonumber
\end{equation}

If $n\geq n_{0}$, then from (\ref{est:sat-lip}), (\ref{th:mb-lip}), and the second condition in (\ref{defn:ep}) it follows that
\begin{eqnarray}
|\psi'(y)-\psi'(x)| & \leq & |f_{0}'(y)-f_{0}'(x)|+\ep_{2}|\varphi_{n_{0},k_{0}}(y)-\varphi_{n_{0},k_{0}}(x)|  
\nonumber  \\[0.5ex]
& \leq & (1-\ep_{1})\Lambda_{n}|y-x|+\ep_{2}2^{(1-\alpha)\beta n_{0}}L_{\varphi}|y-x|  
\nonumber  \\[0.5ex]
& \leq & (1-\ep_{1})\Lambda_{n}|y-x|+\ep_{1}2^{(1-\alpha)\beta n}|y-x|  
\nonumber  \\[0.5ex]
& = & \Lambda_{n}|y-x|-\ep_{1}\left(\Lambda_{n}-2^{(1-\alpha)\beta n}\right)|y-x|  
\nonumber  \\[0.5ex]
& \leq & \Lambda_{n}|y-x|, 
\nonumber
\end{eqnarray}
where in the last inequality we exploited our condition  (\ref{hp:Ln}) on the sequence $\{\Lambda_{n}\}$ (this is the exact point where we need it). This completes the proof that $\psi\in\XX$.

In order to show that the distance between $f_{0}$ and $\psi$ is less than $\ep_{0}$, we exploit (\ref{th:mb-phi-infty}) and (\ref{th:mb-psi-infty}), from which we deduce that
\begin{equation}
|\psi'(x)-f_{0}'(x)|=\ep_{2}|\varphi_{n_{0},k_{0}}(x)|\leq\ep_{2}M_{\varphi}
\qquad
\forall x\in\re
\nonumber
\end{equation}
and
\begin{equation}
|\psi(x)-f_{0}(x)|=\ep_{2}|\psi_{n_{0},k_{0}}(x)|\leq\ep_{2}(k_{0}+1)M_{\varphi}
\qquad
\forall x\in\re.
\nonumber
\end{equation}

At this point the conclusion follows from the last condition in (\ref{defn:ep}).

\subparagraph{\textmd{\textit{Final contradiction -- Proof that $\psi\not\in\CC_{k_{0}}$}}}

Let us assume that $\psi\in\CC_{k_{0}}$, and let $g$ and $z$ be the function and the point corresponding to $\psi$ in the definition of $\CC_{k_{0}}$. Let us set
\begin{equation}
J_{-}:=Z(\psi,g)\cap\left[z-\frac{10}{2^{n_{0}}},z-\frac{8}{2^{n_{0}}}\right]\cap(\re\setminus U_{n_{0}}),
\nonumber
\end{equation}
\begin{equation}
J_{+}:=Z(\psi,g)\cap\left[z+\frac{8}{2^{n_{0}}},z+\frac{10}{2^{n_{0}}}\right]\cap(\re\setminus U_{n_{0}}).
\nonumber
\end{equation}

We claim that both $J_{-}$ and $J_{+}$ have positive Lebesgue measure. Indeed, let us consider the disjoint union
\begin{equation}
J_{3}=J_{-}\cup J_{1}\cup J_{2},
\nonumber
\end{equation}
where we set
\begin{eqnarray}
J_{3} & := & Z(\psi,g)\cap\left[z-\frac{10}{2^{n_{0}}},z+\frac{10}{2^{n_{0}}}\right], 
\nonumber  \\
J_{1} & := & Z(\psi,g)\cap\left[z-\frac{10}{2^{n_{0}}},z-\frac{8}{2^{n_{0}}}\right]\cap U_{n_{0}}, \nonumber  \\
J_{2} & := & Z(\psi,g)\cap\left[z-\frac{8}{2^{n_{0}}},z+\frac{10}{2^{n_{0}}}\right]. 
\nonumber
\end{eqnarray}

Let us estimate the measure of these sets. Due to the first condition  in (\ref{defn:n}), from (\ref{defn:g3}) we know that
\begin{equation}
\meas(J_{3})\geq\frac{19}{2^{n_{0}}}.
\nonumber
\end{equation}

Moreover, $J_{1}\subseteq U_{n_{0}}\cap[z-10/2^{n_{0}},z-8/2^{n_{0}}]$, and since the interval intersects at most three intervals of the form $I_{n_{0},j}$ we deduce that
\begin{equation}
\meas(J_{1})\leq\frac{6}{2^{\beta n_{0}}}.
\nonumber
\end{equation}

Finally, we can estimate the measure of $J_{2}$ with the measure of the interval in its definition, and we deduce that
\begin{equation}
\meas(J_{2})\leq\frac{18}{2^{n_{0}}}.
\nonumber
\end{equation}

From all these inequalities it follows that
\begin{equation}
\meas(J_{-})\geq\frac{19}{2^{n_{0}}}-\frac{18}{2^{n_{0}}}-\frac{6}{2^{\beta n_{0}}},
\nonumber
\end{equation}
which is positive because of the second condition in (\ref{defn:n}). A symmetric argument proves that also the measure of $J_{+}$ is positive.

Now we can argue as in step (VII) of~\cite{kohn}. Let us choose $x\in J_{-}$ and $y\in J_{+}$, with $x$ of density~1 with respect to $J_{-}$, and let us observe that by definition $\psi(x)=g(x)$ and $\psi(y)=g(y)$, but also $\psi'(x)=g'(x)$ because $x$ is of density~1 also for $Z(\psi,g)$. It follows that
\begin{equation}
\psi(y)-\psi(x)-\psi'(x)(y-x)=g(y)-g(x)-g'(x)(y-x).
\label{eqn:psi-g}
\end{equation}

Since $x\not\in U_{n_{0}}$, we deduce also that $x$ does not belong to the support of $\varphi_{n_{0},k_{0}}$, and therefore $\psi'(x)=f_{0}'(x)$. Recalling (\ref{defn:kohn-psi}), from (\ref{eqn:psi-g}) we deduce that
\begin{equation}
\ep_{2}\left(\psi_{n_{0},k_{0}}(y)-\psi_{n_{0},k_{0}}(x)\right)=[g(y)-g(x)-g'(x)(y-x)]-[f(y)-f(x)-f'(x)(y-x)].
\nonumber
\end{equation}

Now we estimate in two different ways the left-hand side and the right-hand side. On the one hand, since $x$ and $y$ are in $[-k_{0},k_{0}]$, and $y-x>3/2^{n_{0}}$, from (\ref{th:mb-gap}) we know that
\begin{equation}
\ep_{2}\left(\psi_{n_{0},k_{0}}(y)-\psi_{n_{0},k_{0}}(x)\strut\right)\geq\frac{\ep_{2}}{2^{(\alpha+1)\beta n_{0}}}.
\nonumber
\end{equation}

On the other hand, since in the interval $[-k_{0},k_{0}]$ the function $g'$ is Lipschitz continuous with constant $k_{0}$, and the function $f_{0}'$ is Lipschitz continuous with constant $L_{0}$, from the mean value theorem we obtain that
\begin{equation}
|g(y)-g(x)-g'(x)(y-x)|\leq k_{0}|y-x|^{2},
\nonumber
\end{equation}
and
\begin{equation}
|f(y)-f(x)-f'(x)(y-x)|\leq L_{0}|y-x|^{2}.
\nonumber
\end{equation}

Recalling that $y-x\leq 20/2^{n_{0}}$, we have proved that
\begin{equation}
\frac{\ep_{2}}{2^{(\alpha+1)\beta n_{0}}}\leq (k_{0}+L_{0})\frac{400}{2^{2n_{0}}},
\nonumber
\end{equation}
which contradicts the last condition in (\ref{defn:n}). This completes the proof.


\setcounter{equation}{0}
\section{Time dependent propagation speed}\label{sec:dgcs}

\subsection{Statement of the problem and previous result}

\paragraph{\textmd{\textit{General functional setting}}}

Let $\mathcal{H}$ be an infinite dimensional Hilbert space, and let $A$ be a nonnegative self-adjoint linear operator on $\mathcal{H}$ with dense domain $D(A)$. We assume that there exist an orthonormal basis $\{e_{i}\}$ of $\mathcal{H}$, and a nondecreasing sequence $\{\lambda_{i}\}$ of nonnegative real numbers such that $\lambda_{i}\to+\infty$ as $i\to +\infty$ and
\begin{equation}
Ae_{i}=\lambda_{i}^{2}e_{i}
\qquad
\forall i\in\n.
\nonumber
\end{equation}

Thanks to the orthonormal basis $\{e_{i}\}$, we can identify elements $u\in \mathcal{H}$ with sequences $\{u_{i}\}\in\ell^{2}$, where $u_{i}:=\langle u,e_{i}\rangle$ is usually called the $i$-th Fourier component of $u$ (here angle brackets denote the scalar product in $\mathcal{H}$).

Given a continuous function $c:[0,+\infty)\to\re$, we consider the abstract evolution equation
\begin{equation}
u''(t)+c(t)Au(t)=0
\qquad
\forall t\geq 0,
\label{eqn:wave}
\end{equation}
with initial data
\begin{equation}
u(0)=u_{0},
\qquad
u'(0)=u_{1}.
\label{eqn:wave-data}
\end{equation}

This equation is an abstract model, for example, of the wave equation $u_{tt}-c(t)\Delta u=0$ with homogeneous Dirichlet or Neumann boundary conditions in a bounded domain $\Omega\subseteq\re^{d}$ with smooth enough boundary. Keeping in mind this concrete model, we often refer to the coefficient $c(t)$ as the ``propagation speed''.

A \emph{very weak solution} to (\ref{eqn:wave})--(\ref{eqn:wave-data}) is a sequence $\{u_{i}(t)\}$ of functions of class $C^{2}$ satisfying the (uncoupled) system of ordinary differential equations
\begin{equation}
u_{i}''(t)+c(t)\lambda_{i}^{2}u_{i}(t)=0
\qquad
\forall i\in\n,\quad\forall t\geq 0,
\label{eqn:wave-ode}
\end{equation}
with initial data
\begin{equation}
u_{i}(0)=u_{0,i},
\qquad
u_{i}'(0)=u_{1,i}.
\label{eqn:ode-data}
\end{equation}

Existence and uniqueness of these very weak solutions is an elementary fact concerning linear ordinary differential equations. In some sense, at this level of generality we are forgetting about the Hilbert space $\mathcal{H}$, the operator $A$, and the abstract equation (\ref{eqn:wave}) in $\mathcal{H}$, and we are just considering the infinite system (\ref{eqn:wave-ode}) of ordinary differential equations depending on a parameter $\lambda_{i}$. Nevertheless, the final goal is interpreting $u_{i}(t)$ as the $i$-th Fourier component of a ``true solution'' to (\ref{eqn:wave})--(\ref{eqn:wave-data}) to be defined as
\begin{equation}
u(t):=\sum_{i=0}^{\infty}u_{i}(t)e_{i}
\qquad
\forall t\geq 0,
\nonumber
\end{equation} 
provided that the series converges in some sense. 

As an extension of the identification of $\mathcal{H}$ with $\ell^{2}$, one can give an abstract definition of Sobolev spaces, Gevrey spaces, distributions, Gevrey ultradistributions with respect to the operator $A$. These spaces are defined as the set of sequences $\{u_{i}\}$ of real numbers such that
\begin{equation}
\sum_{i=0}^{\infty}\varphi(\lambda_{i})u_{i}^{2}<+\infty,
\nonumber
\end{equation}
where $\varphi:[0,+\infty)\to[0,+\infty)$ is a suitable function (a positive/negative power in the case of Sobolev spaces/distributions, and a positive/negative exponential in the case of Gevrey spaces/ultradistributions). In the concrete case where $A$ is the Laplacian in a smooth bounded set with reasonable boundary conditions, these abstract spaces with respect to $A$ coincide with the usual spaces. We refer to~\cite{gg:sns,gg:dgcs-strong,gg:cjs-strong} for a complete set of definitions. Here we just recall the two notions that we need in the sequel.

\begin{defn}[Gevrey spaces/ultradistributions]
\begin{em}

Let $u$ be a sequence $\{u_{i}\}$ of real numbers.
\begin{itemize}

\item  Let $s>0$ and $r>0$ be real numbers. We say that $u$ is a Gevrey function with respect to $A$ of order $s$ and radius $r$, and we write $u\in\GG_{s,r}(A)$, if
\begin{equation}
\|u\|_{\GG_{s,r}(A)}^{2}:=\sum_{i=0}^{\infty}u_{i}^{2}\exp\left(2r\lambda_{i}^{1/s}\right)<+\infty.
\nonumber
\end{equation}

\item  Let $S>0$ and $R>0$ be real numbers. We say that $u$ is a Gevrey ultradistribution with respect to $A$ of order $S$ and radius $R$, and we write $u\in\GG_{-S,R}(A)$, if
\begin{equation}
\|u\|_{\GG_{-S,R}(A)}^{2}:=\sum_{i=0}^{\infty}u_{i}^{2}\exp\left(-2R\lambda_{i}^{1/S}\right)<+\infty.
\nonumber
\end{equation}

\end{itemize}

\end{em}
\end{defn}

\paragraph{\textmd{\textit{Previous result}}}

Now that we have a notion of ``regularity'' for sequences of real numbers, and a notion of very weak solutions to problem (\ref{eqn:wave})--(\ref{eqn:wave-data}), it is natural to ask the following question. Assume that initial data, namely the sequences $\{u_{0,i}\}$ and $\{u_{1,i}\}$ that appear in (\ref{eqn:ode-data}), have a certain regularity, can we conclude that the sequence $\{u_{i}(t)\}$ of solutions to (\ref{eqn:wave-ode})--(\ref{eqn:ode-data}) has the same regularity for all positive times?

The answer depends on the time regularity of the propagation speed. If $c(t)$ is bounded between two positive constants, say
\begin{equation}
0<\mu_{1}\leq c(t)\leq \mu_{2}
\qquad
\forall t\geq 0,
\label{hp:c12}
\end{equation}
and it has $W^{1,1}$ regularity in time (or $BV$ regularity if we admit discontinuous propagation speeds), then the solution has roughly speaking the same regularity of initial data. If $c(t)$ is just \holder\ continuous of some order $\alpha\in(0,1)$, then the regularity in $\GG_{s,r}(A)$ is preserved (with a radius that decreases with time) if $s\leq (1-\alpha)^{-1}$. On the contrary, solutions with initial data in $\GG_{r,s}(A)$ with $s>(1-\alpha)^{-1}$ can exhibit an instantaneous severe derivative loss, in such a way that for all positive times the solution is not even a hyperdistribution of order $S>(1-\alpha)^{-1}$.

Both the positive and the negative result are contained in the seminal paper~\cite{dgcs} (see also~\cite{cjs} for the degenerate case). Here below we quote the negative result, rephrased in the modern language.

\begin{thmbibl}[see~{\cite[Theorem~10]{dgcs}}]\label{thmbibl:dgcs}

Let $\alpha$, $\beta$, $B$ be real numbers with
\begin{equation}
\alpha\in(0,1),
\qquad\qquad
\beta>\frac{1}{1-\alpha},
\qquad\qquad
B>\frac{1}{1-\alpha}.
\nonumber
\end{equation}

Then there exist a function $c:[0,+\infty)\to[1,2]$, \holder\ continuous of order $\alpha$, and a very weak solution $u$ to (\ref{eqn:wave}) such that
\begin{equation}
(u(0),u'(0))\in\GG_{\beta,r}(A)\times\GG_{\beta,r}(A)
\qquad
\forall r>0,
\label{th:dgcs-data}
\end{equation}
but
\begin{equation}
(u(t),u'(t))\not\in\GG_{-B,R}(A)\times\GG_{-B,R}(A)
\qquad
\forall t>0,\quad \forall R>0.
\label{th:dgcs}
\end{equation}

\end{thmbibl}

The derivative loss of the form (\ref{th:dgcs}) is the maximum possible, since we know that with initial data satisfying (\ref{th:dgcs-data}) the solution is at least a Gevrey hyperdistribution of order $S=(1-\alpha)^{-1}$ for all times (see~\cite[Theorem~3]{dgcs}).

\paragraph{\textmd{\textit{The basic ingredient}}}

The basic tool in the construction of the counterexamples in the spirit of Theorem~\ref{thmbibl:dgcs}, as well as in our construction, is considering the three functions
\begin{gather}
\gamma(\ep,t):=1-16\ep^{2}\sin^{4}t-8\ep\sin(2t), \label{defn:dgcs-gamma}  \\
b(\ep,t):=\ep(2t-\sin(2t)), \nonumber \\
w(\ep,t):=\sin t\cdot\exp(b(\ep,t)). 
\label{defn:dgcs-w}
\end{gather}

The fundamental property is that $w(\ep,t)$ satisfies
\begin{equation}
\frac{\partial^{2} w}{\partial t^{2}}(\ep,t)+\gamma(\ep,t)w(\ep,t)=0,
\nonumber
\end{equation}
and grows exponentially with time.

More generally, for every pair of positive real numbers $m$ and $\lambda$, it turns out that the function $v(t):=w(\ep,m\lambda t)$ is a solution to equation
\begin{equation}
v''(t)+\lambda^{2}c(t)v(t)=0,
\label{eqn:ode-cl}
\end{equation}
with time dependent propagation speed $c(t):=m^{2}\gamma(\ep,m\lambda t)$. We observe that for every $\ep\in(0,1)$ the function $c(t)$ satisfies the uniform bounds
\begin{equation}
m^{2}(1-24\ep)\leq c(t)\leq m^{2}(1+24\ep)
\qquad
\forall t\geq 0,
\nonumber
\end{equation}
and the \holder\ condition 
\begin{equation}
|c(t)-c(s)|\leq\ep m^{2}(m\lambda)^{\alpha}H_{\gamma}|t-s|^{\alpha}
\qquad
\forall(s,t)\in[0,+\infty)^{2}
\label{est:gamma-holder}
\end{equation}
for a suitable constant $H_{\gamma}$, and that the solution $v(t)$ grows in time as $\exp(2\ep m\lambda t)$. Now the key point is setting $m\sim 3/2$ and $\ep\sim\lambda^{-\alpha}$, in such a way that when $\lambda$ is large the propagation speed $c(t)$ stays between 1 and 2, its $\alpha$-\holder\ constant remains bounded, but the solution grows in time as $\exp(3\lambda^{1-\alpha}t)$. This is the point where Gevrey spaces and Gevrey ultradistributions of order $(1-\alpha)^{-1}$ come into play.


\subsection{Choice of the functional space and our result}

In this paper we show that the severe derivative loss of Theorem~\ref{thmbibl:dgcs} is the common behavior, namely it happens for the ``generic'' choice of the propagation speed and the initial conditions.

\paragraph{\textmd{\textit{The space of admissible propagation speeds}}}

Let $\alpha$, $\mu_{1}$, $\mu_{2}$, $H$ be real numbers such that
\begin{equation}
\alpha\in(0,1),
\qquad\quad
0<\mu_{1}<\mu_{2},
\qquad\quad
H>0.
\label{hp:accH}
\end{equation}

Let $\FF$ denote the set of functions $c:[0,+\infty)\to[\mu_{1},\mu_{2}]$ such that
\begin{equation}
|c(t)-c(s)|\leq H|t-s|^{\alpha}
\qquad
\forall(t,s)\in[0,+\infty)^{2}.
\label{hp:c-holder}
\end{equation}

The space $\FF$ is a compete metric space with respect to the usual distance
\begin{equation}
\dist_{\FF}(c_{1},c_{2}):=\sup_{t\geq 0}|c_{1}(t)-c_{2}(t)|
\qquad
\forall (c_{1},c_{2})\in\FF^{2}.
\nonumber
\end{equation}

\paragraph{\textmd{\textit{The space of admissible initial velocities}}}

For every $s>0$, let $\GG_{s,\infty}(A)$ denote the set of all sequences $\psi=\{\psi_{i}\}$ of real numbers such that
\begin{equation}
\|\psi\|_{\GG_{s,\infty}(A)}^{2}:=\sum_{i=0}^{\infty}\psi_{i}^{2}\exp\left(\lambda_{i}^{1/s}\log(1+\lambda_{i})\right)<+\infty.
\label{defn:GG-infty}
\end{equation}

It is possible to show that $\GG_{s,\infty}(A)$ is a complete metric space (and actually also a Hilbert space) with respect to the norm $\|\psi\|_{\GG_{s,\infty}(A)}$, and that
\begin{equation}
\GG_{s,\infty}(A)\subseteq\GG_{s,r}(A)
\qquad
\forall r>0.
\nonumber
\end{equation}

\paragraph{\textmd{\textit{Our result}}}

Let us consider the product space $\XX:=\FF\times\GG_{s,\infty}(A)$, which is a complete metric space with respect to the distance
\begin{equation}
\dist_{\XX}((c_{1},\psi_{1}),(c_{2},\psi_{2})):=\dist_{\FF}(c_{1},c_{2})+\|\psi_{1}-\psi_{2}\|_{\GG_{s,\infty}(A)}.
\nonumber
\end{equation}

For every $(c,\psi)\in\XX$, we consider equation (\ref{eqn:wave}) with initial data
\begin{equation}
u(0)=0,\qquad u'(0)=\psi.
\label{eqn:data}
\end{equation}

We are now ready to state our main result of this section.

\begin{thm}\label{thm:dgcs}

Let $\alpha$, $\beta$, $B$ be real numbers as in Theorem~\ref{thmbibl:dgcs}, and let $\mu_{1}$, $\mu_{2}$ and $H$ be three positive real numbers satisfying (\ref{hp:accH}). Let $\XX$ be the space defined above with $s=\beta$.

Then the set of all pairs $(c,\psi)\in\XX$ such that the corresponding solution to problem (\ref{eqn:wave})--(\ref{eqn:data}) satisfies (\ref{th:dgcs}) is residual in $\XX$.

\end{thm}

We conclude with some comments and an open problem in the critical case.

\begin{rmk}
\begin{em}

The choice to set one initial condition equal to zero in (\ref{eqn:data}) is just to emphasize that the pathology originates from the lack of regularity of the propagation speed, and not from a special choice of initial data. An analogous result holds true with minimal changes with initial data such as $u(0)=\psi$ and $u'(0)=0$, or $u(0)=\psi$ and $u'(0)=\psi_{0}$ where $\psi_{0}$ is a given fixed value, or with both initial conditions that are allowed to vary.

We also decided to work in the space $\GG_{s,\infty}(A)$ because data in this space are more regular than data in $\GG_{s,r_{0}}(A)$, but the same argument works with any fixed value of $r_{0}$. Finally, we observe that in (\ref{defn:GG-infty}) we can replace $\log(1+\lambda_{i})$ with any function that tends to $+\infty$ slower than any power.

\end{em}
\end{rmk}

\begin{rmk}
\begin{em}

As for the choice of the functional space for the propagation speed, we decided to put both conditions (\ref{hp:c-holder}) and  (\ref{hp:c12}) in order to emphasize that the derivative loss is possible also if $H$ is very small, and $\mu_{1}$ and $\mu_{2}$ are close to each other. 

On the other hand, this technique is quite flexible, and with minimal changes we can work in a space without any bound on the \holder\ constant, and even any bound from above on $c(t)$. In this case, however, we need to consider the full norm in the space of \holder\ functions, and not just the uniform norm, in order to have completeness.

\end{em}
\end{rmk}

\begin{rmk}
\begin{em}

Several extensions to Theorem~\ref{thmbibl:dgcs} and Theorem~\ref{thm:dgcs} are possible with minimal technical adjustments. For example, in Theorem~\ref{thmbibl:dgcs} we can ask that initial conditions satisfy (\ref{th:dgcs-data}) for every $\beta>(1-\alpha)^{-1}$, and the corresponding solution satisfies (\ref{th:dgcs}) for every $B>(1-\alpha)^{-1}$. One can consider also more general continuity moduli, and not just \holder\ continuity. For these questions, we refer to~\cite{gg:sns,gg:dgcs-strong}.

\end{em}
\end{rmk}

Let us consider the critical case where the propagation speed $c(t)$ is \holder\ continuous of order $\alpha$, and initial conditions are in $\GG_{s,r_{0}}(A)$ with order $s$ exactly equal to $(1-\alpha)^{-1}$ and some finite radius $r_{0}>0$. In this case the regularity result of~\cite[Theorem~2]{dgcs} guarantees that the solution remains in Gevrey spaces with the same order $s$ (and radius decreasing with time) only in a finite time interval $[0,t_{0}]$, with $t_{0}>0$ depending on $r_{0}$. The known counterexamples do not address the critical case, and this motivates the following question.

\begin{open}
\begin{em}

Is it possible to find a propagation speed $c(t)$ that is \holder\ continuous of some order $\alpha\in(0,1)$, an initial velocity $\psi\in\GG_{s,r_{0}}(A)$ with $s=(1-\alpha)^{-1}$ and some $r_{0}>0$, and a positive real number $t_{0}$, such that the corresponding solution to problem (\ref{eqn:wave})--(\ref{eqn:data}) satisfies (\ref{th:dgcs}), or any equivalent form of derivative loss, for every $t>t_{0}$?

\end{em}
\end{open}


\subsection{Technical preliminaries}

For the convenience of the reader, in this subsection we state and prove the basic energy estimate for solutions to an ordinary differential equation of the form (\ref{eqn:ode-cl}) with Lipschitz continuous coefficient $c(t)$. The same argument provides both an estimate from below and an estimate from above. In this paper, however, we need only the first one.

\begin{lemma}\label{lemma:dgcs-energy}

Let $v:[0,+\infty)\to\re$ be a solution to an equation of the form (\ref{eqn:ode-cl}), where $\lambda$ is a positive real number, and the propagation speed $c(t)$ satisfies the uniform bound (\ref{hp:c12}) and it is Lipschitz continuous with constant $L$. Let us set
\begin{equation}
\mu_{3}:=\min\{1,\mu_{1}\}\cdot\min\left\{1,\frac{1}{\mu_{2}}\right\},
\qquad
\mu_{4}:=\frac{L}{\mu_{1}},
\qquad
\mu_{5}:=\max\left\{1,\frac{1}{\mu_{1}}\right\}\cdot\max\{1,\mu_{2}\}.
\nonumber
\end{equation}

Then for every $0\leq t_{0}\leq t$ it turns out that
\begin{equation}
|v'(t)|^{2}+\lambda^{2}|v(t)|^{2}\geq \mu_{3}\left(|v'(t_{0})|^{2}+\lambda^{2}|v(t_{0})|^{2}\right)\exp(-\mu_{4}(t-t_{0})),
\nonumber
\end{equation}
and
\begin{equation}
|v'(t)|^{2}+\lambda^{2}|v(t)|^{2}\leq \mu_{5}\left(|v'(t_{0})|^{2}+\lambda^{2}|v(t_{0})|^{2}\right)\exp(\mu_{4}(t-t_{0})).
\nonumber
\end{equation}

\end{lemma}

\begin{proof}

Let us consider the classical energies
\begin{equation}
E(t):=|v'(t)|^{2}+\lambda^{2}|v(t)|^{2}
\qquad\mbox{and}\qquad
F(t):=|v'(t)|^{2}+\lambda^{2}c(t)|v(t)|^{2},
\nonumber
\end{equation}
sometimes called the ``Kovaleskyan'' and the ``hyperbolic'' energy. From the uniform bounds (\ref{hp:c12}) it follows that
\begin{equation}
\min\left\{1,\mu_{2}^{-1}\right\}F(t)\leq E(t)\leq\max\left\{1,\mu_{1}^{-1}\right\}F(t)
\qquad
\forall t\geq 0.
\label{EF-equiv}
\end{equation}

Assuming for a while that $c(t)$ is of class $C^{1}$, the hyperbolic energy is of class $C^{1}$ as well, and its time-derivative is
\begin{equation}
F'(t)=c'(t)\lambda^{2}|v(t)|^{2}
\qquad
\forall t\geq 0,
\nonumber
\end{equation}
and hence
\begin{equation}
|F'(t)|\leq\frac{|c'(t)|}{c(t)}\cdot c(t)\lambda^{2}|v(t)|^{2}\leq\frac{L}{\mu_{1}}F(t)
\qquad
\forall t\geq 0.
\label{est:F'}
\end{equation}

Integrating this differential inequality we deduce that
\begin{equation}
F(t_{0})\exp\left(-\frac{L}{\mu_{1}}(t-t_{0})\right)\leq F(t)\leq F(t_{0})\exp\left(\frac{L}{\mu_{1}}(t-t_{0})\right)
\qquad
\forall t\geq t_{0}.
\nonumber
\end{equation}

Combining with (\ref{EF-equiv}), we obtain the required estimates.

If $c(t)$ is just Lipschitz continuous, we obtain the same result through an approximation procedure, or by observing that $F(t)$ is Lipschitz continuous and its derivative satisfies (\ref{est:F'}) for almost every $t\geq 0$.
\end{proof}


\subsection{Proof of Theorem~\ref{thm:dgcs}}

\paragraph{\textmd{\textit{Quantitative non-pathological behavior}}}

Let $\CC$ denote the set of ``non-counterexamples'', namely the set of all pairs $(c,\psi)\in\XX$ such that the corresponding solution to problem (\ref{eqn:wave})--(\ref{eqn:data}) satisfies
\begin{equation}
\exists t>0\quad\exists R>0
\qquad
(u(t),u'(t))\in\GG_{-B,R}(A)\times\GG_{-B,R}(A).
\label{defn:dgcs-CC}
\end{equation}

We have to show that $\CC$ is a countable union of closed sets with empty interior. To this end, we state in a quantitative way the definition of $\CC$. For every positive integer $k$, we consider the set $\CC_{k}$ of all pairs $(c,\psi)\in\XX$ such that the corresponding solution to (\ref{eqn:wave})--(\ref{eqn:data}) satisfies
\begin{equation}
\exists t\in\left[1/k,k\right]
\qquad\quad
\|u(t)\|_{\GG_{-B,k}(A)}^{2}+\|u'(t)\|_{\GG_{-B,k}(A)}^{2}\leq k.
\nonumber
\end{equation}

In other words, now $t$ is confined in a compact set away from~0, we have fixed $R=k$, and we have prescribed a bound on the norm of $(u(t),u'(t))$ in the space of Gevrey ultradistributions $\GG_{-B,k}(A)$. 

The proof is complete if we show that $\CC$ is the union of all $\CC_{k}$'s, and that $\CC_{k}$ is a closed set with empty interior for every positive integer $k$.

\paragraph{\textmd{\textit{The set $\CC$ is the union of all $\CC_{k}$'s}}}

Let $(c,\psi)\in\CC$, so that the corresponding solution to problem (\ref{eqn:wave})--(\ref{eqn:data}) satisfies (\ref{defn:dgcs-CC}). Then it turns out that $(c,\psi)\in\CC_{k}$ provided that we choose $k$ such that
\begin{equation}
\frac{1}{k}\leq t\leq k,
\qquad
R\leq k,
\qquad
\|u(t)\|_{\GG_{-B,k}(A)}^{2}+\|u'(t)\|_{\GG_{-B,k}(A)}^{2}\leq k.
\nonumber
\end{equation}

To this end, we just need to observe that every $u\in\GG_{-B,R}(A)$ belongs also to the space $\GG_{-B,R'}(A)$ for every $R'\geq R$, and $\|u\|_{\GG_{-B,R'}(A)}\leq\|u\|_{\GG_{-B,R}(A)}$.

\paragraph{\textmd{\textit{The set $\CC_{k}$ is closed}}}

Let $k$ be a \emph{fixed} positive integer. Let $\{(c_{n},\psi_{n})\}\subseteq\CC_{k}$ be any sequence, and let us assume that $c_{n}(t)\to c_{\infty}(t)$ uniformly in $[0,+\infty)$, and that $\psi_{n}\to\psi_{\infty}$ in $\GG_{\beta,\infty}(A)$ (and in particular $\psi_{n}\to\psi_{\infty}$ in the component-wise sense). We claim that $(c_{\infty},\psi_{\infty})\in\CC_{k}$.

Let $u_{n}(t)$ denote the very weak solution to (\ref{eqn:wave})--(\ref{eqn:data}) with $c:=c_{n}$ and $\psi:=\psi_{n}$. From the definition of $\CC_{k}$ we know that for every $n$ there exists $t_{n}\in[1/k,k]$ such that
\begin{equation}
\|u_{n}(t_{n})\|_{\GG_{-B,k}(A)}^{2}+\|u_{n}'(t_{n})\|_{\GG_{-B,k}(A)}^{2}\leq k.
\label{est:dgcs-un}
\end{equation}

Up to subsequences (not relabeled) we can always assume that $t_{n}\to t_{\infty}\in[1/k,k]$. Let $u_{\infty}$ be the solution to problem  (\ref{eqn:wave})--(\ref{eqn:data}) with $c:=c_{\infty}$ and $\psi:=\psi_{\infty}$. Since solutions to linear ordinary differential equation of the form (\ref{eqn:wave-ode}) depend in a continuous way on the coefficient $c(t)$ and on initial data, we deduce that $u_{n}\to u_{\infty}$ in the component-wise sense, namely that
\begin{equation}
\langle u_{n}(t),e_{i}\rangle\to\langle u_{\infty}(t),e_{i}\rangle
\qquad
\forall i\in\n
\nonumber
\end{equation}
uniformly on compact sets (here, with a little abuse of notation, we used the scalar product with $e_{i}$ in order to denote the $i$-th element of a sequence that does not necessarily belong to $\ell^{2}$). Since the norm in $\GG_{-B,k}(A)$ is lower semicontinuous with respect to component-wise convergence, we can pass (\ref{est:dgcs-un}) to the limit and deduce that
\begin{equation}
\|u_{\infty}(t_{\infty})\|_{\GG_{-B,k}(A)}^{2}+\|u_{\infty}'(t_{\infty})\|_{\GG_{-B,k}(A)}^{2}\leq k,
\nonumber
\end{equation}
which completes the proof that $(c_{\infty},\psi_{\infty})\in\CC_{k}$.

\paragraph{\textmd{\textit{The set $\CC_{k}$ has empty interior}}}

Let us assume that there exist an integer $k_{0}\geq 1$, a pair $(c_{0},\psi_{0})\in\XX$, and a real number $\ep_{0}>0$ such that $B_{\XX}((c_{0},\psi_{0}),\ep_{0})\subseteq\CC_{k_{0}}$. 

\subparagraph{\textmd{\textit{Regularization of the center}}}

Up to a small modification of $c_{0}$, and a small reduction of the radius $\ep_{0}$, we can assume that $c_{0}$ has the following further properties.
\begin{itemize}

\item  It is Lipschitz continuous with some constant $L_{0}$.

\item  It is constant in the interval $[0,\delta]$ for some $\delta\in(0,1/k_{0})$.

\item  It does not saturate the inequalities in (\ref{hp:c12}) and (\ref{hp:c-holder}), namely there exists $\ep_{1}>0$ such that
\begin{equation}
\mu_{1}+\ep_{1}\leq c_{0}(t)\leq \mu_{2}-\ep_{1}
\qquad
\forall t\geq 0,
\label{sat:cn}
\end{equation}
and
\begin{equation}
|c_{0}(t)-c_{0}(s)|\leq(1-\ep_{1})H|t-s|^{\alpha}
\qquad
\forall(t,s)\in[0,+\infty)^{2}.
\label{sat:cn-holder}
\end{equation}

\end{itemize}

\subparagraph{\textmd{\textit{Use of rescaled basic ingredient}}}

Let us modify $c_{0}$ in the initial interval $[0,\delta]$, and let us modify $\psi_{0}$ in just one Fourier component, in such a way that the modified pair still belongs to $B_{\XX}((c_{0},\psi_{0}),\ep_{0})$, but it does not belong to $\CC_{k_{0}}$. This would give a contradiction.

In order to modify $c_{0}$, we choose $m:=[c_{0}(0)]^{1/2}$, and for every positive integer $n$ we set
\begin{equation}
\ep_{n}:=\frac{\ep_{1}H}{m^{\alpha+2}H_{\gamma}}\cdot\frac{1}{\lambda_{n}^{\alpha}},
\qquad\qquad
\delta_{n}:=\frac{2\pi}{m\lambda_{n}}\left\lfloor\frac{m\lambda_{n}\delta}{2\pi}\right\rfloor,
\label{defn:ep-delta}
\end{equation}
where $H_{\gamma}$ is the constant that appears in (\ref{est:gamma-holder}). Then we consider the function
\begin{equation}
c_{n}(t):=\left\{
\begin{array}{l@{\qquad}l}
m^{2}\gamma(\ep_{n},m\lambda_{n} t) & \mbox{if }0\leq t\leq\delta_{n}, \\[0.5ex]
c_{0}(t) & \mbox{if }t\geq\delta_{n}.
\end{array}\right.
\nonumber
\end{equation}

Since $\delta_{n}\leq\delta$, and $m\lambda_{n}\delta_{n}$ is an integer multiple of $2\pi$, from (\ref{defn:dgcs-gamma}) we obtain that
\begin{equation}
m^{2}\gamma(\ep_{n},m\lambda_{n}\delta_{n})=m^{2}=c_{0}(0)=c_{0}(\delta_{n}),
\nonumber
\end{equation}
and therefore the function $c_{n}(t)$ is well-defined and continuous in the half-line $t\geq 0$.

In order to modify $\psi_{0}$, we consider the vector $\psi_{n}\in\GG_{\beta,\infty}(A)$ obtained from $\psi_{0}$ by replacing the $n$-th Fourier component with 
$$\frac{1}{1+\lambda_{n}}\cdot\exp\left(-\lambda_{n}^{1/\beta}\log(1+\lambda_{n})\right),$$
and leaving the other components unchanged. In this way it turns out that $\psi_{n}\to\psi_{0}$ in $\GG_{\beta,\infty}(A)$.

\subparagraph{\textmd{\textit{Final contradiction: $(c_{n},\psi_{n})\in B_{\XX}((c_{0},\psi_{0}),\ep_{0})$ for $n$ large enough}}}

Let us check that $c_{n}\in\FF$ for $n$ sufficiently large, and $c_{n}\to c_{0}$ in $\FF$. First of all, we observe that $\ep_{n}\to 0$ as $n\to +\infty$, and therefore $c_{n}\to c_{0}$ uniformly in $[0,+\infty)$. Due to (\ref{sat:cn}), it follows that $c_{n}(t)$ satisfies the uniform bounds (\ref{hp:c12}) when $n$ is large enough. It remains to show that $c_{n}$ is \holder\ continuous of order $\alpha$ with constant $H$. Due to (\ref{sat:cn-holder}), it is enough to show that $c_{n}-c_{0}$ is \holder\ continuous of order $\alpha$ with constant $\ep_{1}H$. To this end, we recall that $m^{2}=c_{0}(0)=c_{0}(t)$ for every $t\in[0,\delta_{n}]$, and therefore  
\begin{equation}
c_{n}(t)-c_{0}(t)=\left\{
\begin{array}{l@{\qquad}l}
m^{2}\gamma(\ep_{n},m\lambda_{n} t)-m^{2} & \mbox{if }0\leq t\leq\delta_{n}, \\[0.5ex]
0 & \mbox{if }t\geq\delta_{n}.
\end{array}\right.
\nonumber
\end{equation}

At this point, the conclusion follows from (\ref{est:gamma-holder}) and our definition (\ref{defn:ep-delta}) of $\ep_{n}$.

Since we already observed that $\psi_{n}\to\psi_{0}$ in $\GG_{\beta,\infty}(A)$, this shows that $(c_{n},\psi_{n})\in B_{\XX}((c_{0},\psi_{0}),\ep_{0})$ when $n$ is large enough.

\subparagraph{\textmd{\textit{Final contradiction: $(c_{n},\psi_{n})\not\in C_{k_{0}}$ for $n$ large enough}}}

Let us consider the solution $u_{n}$ to problem (\ref{eqn:wave})--(\ref{eqn:data}) with $c:=c_{n}$ and $\psi:=\psi_{n}$. In order to estimate from below the norm of $u_{n}(t)$, we just consider its $n$-th Fourier component $v_{n}(t):=\langle u_{n}(t),e_{n}\rangle$, which is a solution to the scalar ordinary differential equation
\begin{equation}
v_{n}''(t)+c_{n}(t)\lambda_{n}^{2}v_{n}(t)=0,
\nonumber
\end{equation}
with initial data
\begin{equation}
v_{n}(0)=0,
\qquad
v_{n}'(0)=\frac{1}{1+\lambda_{n}}\exp\left(-\lambda_{n}^{1/\beta}\log(1+\lambda_{n})\right).
\nonumber
\end{equation}

In the interval $[0,\delta_{n}]$ we know that 
\begin{equation}
v_{n}(t)=\frac{1}{m\lambda_{n}}\cdot\frac{1}{1+\lambda_{n}}\exp\left(-\lambda_{n}^{1/\beta}\log(1+\lambda_{n})\right)\cdot w(\ep_{n},m\lambda_{n}t),
\nonumber
\end{equation}
where $w$ is defined by (\ref{defn:dgcs-w}).  Recalling that $m\lambda_{n}\delta_{n}$ is an integer multiple of $2\pi$, from this explicit expression we deduce that $v_{n}(\delta_{n})=0$ and
\begin{equation}
v_{n}'(\delta_{n})=\frac{1}{1+\lambda_{n}}\exp\left(-\lambda_{n}^{1/\beta}\log(1+\lambda_{n})\right)\cdot\exp\left(2\ep_{n}m\lambda_{n}\delta_{n}\right).
\nonumber
\end{equation}

From (\ref{defn:ep-delta}) it turns out that, when $n$ is large enough, $2\delta_{n}\geq\delta$ and $\ep_{n}m\delta\geq 2r_{0}\lambda_{n}^{-\alpha}$, where $r_{0}$ is a positive constant independent of $n$. It follows that
\begin{eqnarray}
v_{n}'(\delta_{n}) & \geq & \exp\left(2r_{0}\lambda_{n}^{1-\alpha}-\lambda_{n}^{1/\beta}\log(1+\lambda_{n})-\log(1+\lambda_{n})\right)  \nonumber
\\[0.5ex]
& \geq & \exp\left(r_{0}\lambda_{n}^{1-\alpha}-\lambda_{n}^{1/\beta}\log(1+\lambda_{n})\right),
\nonumber
\end{eqnarray}
where again in the last inequality we exploited that $n$ is large enough.

For $t\geq\delta_{n}$ we know that $c_{n}(t)$ coincides with $c_{0}(t)$, and therefore it is Lipschitz continuous with constant $L_{0}$. This allows to apply Lemma~\ref{lemma:dgcs-energy}, from which we deduce that
\begin{eqnarray}
|v_{n}'(t)|^{2}+\lambda_{n}^{2}|v_{n}(t)|^{2} & \geq & \mu_{3}\left(|v_{n}'(\delta_{n})|^{2}+\lambda_{n}^{2}|v_{n}(\delta_{n})|^{2}\right)\exp\left(-\mu_{4}(t-\delta_{n})\strut\right) \nonumber
\\[0.5ex]
& \geq & \mu_{3}\exp\left(2r_{0}\lambda_{n}^{1-\alpha}-2\lambda_{n}^{1/\beta}\log(1+\lambda_{n})-\mu_{4}t\right),
\nonumber
\end{eqnarray}
where $\mu_{3}$ and $\mu_{4}$ are two constants independent of $n$.

Since $1/k_{0}\geq\delta\geq\delta_{n}$, this estimate holds true for every $t\in[1/k_{0},k_{0}]$, and in particular
\begin{eqnarray}
|v_{n}'(t)|^{2}+|v_{n}(t)|^{2} & \geq & \frac{1}{\lambda_{n}^{2}}\left(|v_{n}'(t)|^{2}+\lambda_{n}^{2}|v_{n}(t)|^{2}\right) \nonumber  \\
& \geq & \mu_{3}\exp\left(2r_{0}\lambda_{n}^{1-\alpha}-2\lambda_{n}^{1/\beta}\log(1+\lambda_{n})-\mu_{4}k_{0}-2\log\lambda_{n}\right) \nonumber  \\[0.5ex]
& \geq & \mu_{3}\exp\left(r_{0}\lambda_{n}^{1-\alpha}-2\lambda_{n}^{1/\beta}\log(1+\lambda_{n})\right)  \nonumber
\end{eqnarray}
for every $t\in[1/k_{0},k_{0}]$, where again the last inequality holds true when $n$ is sufficiently large. From this last estimate we conclude that
\begin{equation}
\|u_{n}'(t)\|_{\GG_{-B,k_{0}}(A)}^{2}+\|u_{n}(t)\|_{\GG_{-B,k_{0}}(A)}^{2}\geq \mu_{3}\exp\left(r_{0}\lambda_{n}^{1-\alpha}-2\lambda_{n}^{1/\beta}\log(1+\lambda_{n})-2k_{0}\lambda_{n}^{1/B}\right)
\nonumber
\end{equation}
for every $t\in[1/k_{0},k_{0}]$. Since $1/\beta<1-\alpha$ and $1/B<1-\alpha$, the right-hand side tends to $+\infty$ as $n\to +\infty$, and therefore the left-hand side is larger than $k_{0}$ for $n$ large enough. This shows that $(c_{n},\psi_{n})\not\in C_{k_{0}}$ when $n$ is sufficiently large, which completes the proof.


\setcounter{equation}{0}
\section{A continuity/transport equation}\label{sec:alberti}

\subsection{Statement of the problem and previous result}

\paragraph{\textmd{\textit{General setting}}}

Let us consider the transport equation
\begin{equation}
\partial_{t}\rho+u\cdot\nabla\rho=0
\qquad
t\geq 0,\quad x\in\re^{d},
\label{eqn:transport}
\end{equation}
where $\partial_{t}$ denotes the partial derivative with respect to the time variable $t$, the dot denotes the scalar product in $\re^{d}$, and $\nabla\rho$ denotes the gradient of the scalar function $\rho$ with respect to the space variables $x\in\re^{d}$. The so-called velocity field $u$ is a given function $u:[0,+\infty)\times\re^{d}\to\re^{d}$ satisfying the divergence-free condition
\begin{equation}
\div u=0
\qquad
\forall t\geq 0,\quad\forall x\in\re^{d},
\label{eqn:div-free}
\end{equation}
in some sense (again $\div$ denotes the divergence with respect to the space variables), and one looks for a solution $\rho:[0,+\infty)\times\re^{d}\to\re$ to (\ref{eqn:transport}) that satisfies the initial condition
\begin{equation}
\rho(0,x)=\theta(x)
\qquad
\forall x\in\re^{d},
\label{eqn:transport-data}
\end{equation}
where $\theta:\re^{d}\to\re$ is a given function. Due to (\ref{eqn:div-free}), the transport equation can be rewritten as $\partial_{t}u+\div(\rho u)=0$, and in this form it is usually called a continuity equation.

\paragraph{\textmd{\textit{Previous result}}}

If the velocity field is Lipschitz continuous with respect to the space variables, uniformly with respect to time, then problem (\ref{eqn:transport})--(\ref{eqn:transport-data}) has a unique solution, and roughly speaking this solution has the same space regularity of the initial datum $\theta$, with an estimate of the norms. This follows from the classical method of characteristics, even without the divergence-free condition.

When the velocity field is not Lipschitz continuous, a well-posedness theory for weak solutions has been developed, provided that $u$ has some Sobolev or $BV$ regularity in the space variables, and satisfies (\ref{eqn:div-free}) or at least some bound from below on the divergence in order to prevent concentration phenomena. For more details, we refer to the seminal papers~\cite{diperna-lions,ambrosio}, and to the recent survey~\cite{AC:survey}.

Nevertheless, these weak solutions can exhibit a severe derivative loss, as shown in the recent result quoted below, where the velocity field $u$ belongs to $W^{1,p}(\re^{d})$ for every $p\geq 1$ (but not for $p=+\infty$, of course) uniformly in time, the initial condition $\theta$ is of class $C^{\infty}$ with compact support, and for all positive times the solution $\rho$ ``has lost all its space derivatives'', in the sense that it does not belong to the space $H^{s}(\re^{d})$ for every positive (real) number $s$. We refer to~\cite{alberti:controesempio} and to the references quoted therein for the theory of the fractional Sobolev spaces $H^{s}(\re^{d})$ and their homogeneous counterparts $\dot{H}^{s}(\re^{d})$, as well as for the definition of the norm in these spaces. In this section $B_{d}(x,r)$ denotes the ball with center in a point $x\in\re^{d}$ and radius $r$.

\begin{thmbibl}[see~{\cite[Theorem~1]{alberti:controesempio}}]\label{thmbibl:alberti}

For every integer $d\geq 2$ there exist a measurable vector field $u:[0,+\infty)\times\re^{d}\to\re^{d}$, and a measurable solution $\rho:[0,+\infty)\times\re^{d}\to\re$ to the transport equation (\ref{eqn:transport}), satisfying the following conditions.
\begin{itemize}

\item  \emph{(Compact support in space).} There exists $R>0$ such that
\begin{equation}
\|u(t,x)\|+|\rho(t,x)|=0
\qquad
\forall t\geq 0,\quad\forall x\in\re^{d}\setminus B_{d}(0,R).
\nonumber
\end{equation}

\item  \emph{(Global boundedness).} There exists a real number $M>0$ such that
\begin{equation}
\|u(t,x)\|+|\rho(t,x)|\leq M
\qquad
\forall t\geq 0,\quad\forall x\in\re^{d}.
\nonumber
\end{equation}

\item  \emph{(Sobolev regularity in space of the velocity field).} For every $p\in[1,+\infty)$ there exists a real constant $M_{p}$ such that
\begin{equation}
\|u(t,x)\|_{W^{1,p}(\re^{d})}\leq M_{p}
\qquad
\forall t\geq 0.
\nonumber
\end{equation}

\item  \emph{(Divergence-free condition).} The vector field $u$ satisfies (\ref{eqn:div-free}) as an equality in $L^{p}(\re^{d})$ for every $t\geq 0$.

\item  \emph{(Smoothness of initial data).} The initial condition $\theta(x):=\rho(0,x)$ is of class $C^{\infty}$ in $\re^{d}$.

\item  \emph{(Severe derivative loss for positive times).} For every $t>0$ and every $s>0$, the function $x\to\rho(t,x)$ does not belong to $H^{s}(\re^{d})$.

\end{itemize} 

\end{thmbibl}

We observe that in the previous result the loss of regularity is localized in a neighborhood of a point, in the sense that both $u$ and $\rho$ are actually of class $C^{\infty}$ for $t\geq 0$ and $x\neq 0$. We observe also that the result is optimal in the sense that it is known that some derivative ``of logarithmic order'' survives during the evolution (for more details, we refer to~\cite{brue-nguyen:log-der} and to the references quoted therein).

\paragraph{\textmd{\textit{The basic ingredient}}}

The proof of Theorem~\ref{thmbibl:alberti} is based on the following result, which we also need in our construction.

\begin{thmbibl}[see~{\cite{alberti:exp}} and~{\cite[Theorem~8 and Remark~9]{alberti:controesempio}}]\label{thmbibl:alberti-basic}

For every integer $d\geq 2$ there exist a divergence-free vector field $u_{*}:[0,+\infty)\times\re^{d}\to\re^{d}$ of class $C^{\infty}$, and an initial condition $\theta_{*}:\re^{d}\to\re$ of class $C^{\infty}$ with compact support, such that the corresponding solution $\rho_{*}:[0,+\infty)\times\re^{d}\to\re$ to problem (\ref{eqn:transport})--(\ref{eqn:transport-data}) is of class $C^{\infty}$ and satisfies the following conditions.
\begin{itemize}

\item  \emph{(Compact support in space).} There exists $R>0$ such that
\begin{equation}
\|u_{*}(t,x)\|+|\rho_{*}(t,x)|=0
\qquad
\forall t\geq 0,\quad\forall x\in\re^{d}\setminus B_{d}(0,R).
\nonumber
\end{equation}

\item  \emph{(Global boundedness).} There exists a real number $M>0$ such that
\begin{equation}
\|u_{*}(t,x)\|+\|D_{x} u_{*}(t,x)\|+|\rho_{*}(t,x)|\leq M
\qquad
\forall t\geq 0,\quad\forall x\in\re^{d}.
\label{th:bounds-ur}
\end{equation}

\item  \emph{(Exponential growth of homogeneous Sobolev norms).} There exists a real number $c>0$ with the following property: for every real number $s\in(0,2)$ there exists a real number $C_{s}>0$ such that
\begin{equation}
\|\rho_{*}(t,x)\|_{\dot{H}^{s}(\re^{d})}\geq C_{s}\exp(cst)
\qquad
\forall t\geq 0.
\label{th:rho-growth}
\end{equation}

\end{itemize}

\end{thmbibl}


\subsection{Choice of the functional space and our result}

In this paper we show that the derivative loss of Theorem~\ref{thmbibl:alberti} is the common behavior, namely it happens for the ``generic'' choice of the velocity field $u$ and of the smooth initial condition $\theta$.

\paragraph{\textmd{\textit{The space of admissible initial data}}}

Let $\FF$ denote the set of all functions $\theta:\re^{d}\to\re$ of class $C^{\infty}$ such that $\theta(x)=0$ for every $x\in\re^{d}$ with $\|x\|\geq 1$. For every integer $k\geq 0$ we set
\begin{equation}
\|\theta\|_{k,\infty}:=\sum_{|\alpha|=k}\sup_{\|x\|\leq 1}\left|\partial^{\alpha}\theta(x)\right|,
\nonumber
\end{equation}
where the sum ranges over all multi-indices $\alpha=(\alpha_{1},\ldots,\alpha_{d})\in\n^{d}$ with $|\alpha|:=\alpha_{1}+\cdots+\alpha_{d}=k$. Finally, we set
\begin{equation}
\distF(\theta_{1},\theta_{2}):=\sum_{k=0}^{\infty}\frac{1}{2^{k}}\arctan\left(\|\theta_{1}-\theta_{2}\|_{k,\infty}\right).
\label{defn:dist-C-infty}
\end{equation}

It is possible to show that this formula defines a distance on $\FF$, and that $\FF$ is a complete metric space with respect to this distance. Moreover, since the series in (\ref{defn:dist-C-infty}) is dominated by a converging series, one can show that $\theta_{n}\to\theta_{\infty}$ in $\FF$ if and only if $\partial^{\alpha}\theta_{n}\to\partial^{\alpha}\theta_{\infty}$ uniformly in $\re^{d}$ for every multi-index $\alpha$.

\paragraph{\textmd{\textit{The space of admissible velocity fields}}}

Let $\VV$ denote the set of all continuous vector fields $u:[0,+\infty)\times\re^{d}\to\re^{d}$ satisfying the following conditions.
\begin{itemize}

\item[(V1)] (Compact support in space). It turns out that $u(t,x)=0$ for every $t\geq 0$ and every $x\in\re^{d}$ with $\|x\|\geq 1$.

\item[(V2)] (Global boundedness). It turns out that $\|u(t,x)\|\leq 1$ for every $t\geq 0$ and every $x\in\re^{d}$.

\item[(V3)] (Sobolev regularity in space). For every $t\geq 0$ the function $x\to u(t,x)$ belongs to $W^{1,p}(\re^{d})$ for every $p\in[1,+\infty)$ and
\begin{equation}
\|D_{x} u(t,x)\|_{L^{p}(\re^{d})}\leq p^{4}
\qquad
\forall t\geq 0,\quad\forall p\geq 1.
\label{hp:p4}
\end{equation}

\item[(V4)] (Divergence-free condition). The vector field $u$ satisfies (\ref{eqn:div-free}) as an equality in $L^{p}(\re^{d})$ for every $t\geq 0$.

\end{itemize}

We observe that, due to Sobolev imbeddings, for every $t\geq 0$ the function $x\to u(t,x)$ is actually \holder\ continuous in $\re^{d}$ of every order $\alpha\in(0,1)$ (but of course it is not necessarily Lipschitz continuous). We consider in $\VV$ the usual distance
\begin{equation}
\dist_{\VV}(u_{1},u_{2}):=\sup\{\|u_{1}(t,x)-u_{2}(t,x)\|:t\geq 0,\ x\in\re^{d}\}
\qquad
\forall (u_{1},u_{2})\in\VV^{2},
\label{defn:dist-VV}
\end{equation}
 
Since the norms in the left-hand side of (\ref{hp:p4}) are lower semicontinuous with respect to uniform convergence, and also the divergence-free condition passes to the limit, it turns out that
$\VV$ is a complete metric space with respect to the distance (\ref{defn:dist-VV}).

\paragraph{\textmd{\textit{Our result}}}

Let us consider the product space $\XX:=\VV\times\FF$, which is a complete metric space with respect to the distance
\begin{equation}
\dist_{\XX}((u_{1},\theta_{1}),(u_{2},\theta_{2})):=\dist_{\VV}(u_{1},u_{2})+\dist_{\FF}(\theta_{1},\theta_{2}).
\nonumber
\end{equation}

We are now ready to state our main result of this section.

\begin{thm}\label{thm:transport}

Let $\XX$ be the space defined above.

Then the set of pairs $(u,\theta)\in\XX$ such that the corresponding solution $\rho$ to problem (\ref{eqn:transport})--(\ref{eqn:transport-data}) satisfies the severe derivative loss for positive times as in the last statement of Theorem~\ref{thmbibl:alberti} is residual in $\XX$.

\end{thm}

We conclude with some comments on the result, and one open problem concerning the localization in space of the derivative loss (see the comment after Theorem~\ref{thmbibl:alberti}).

\begin{rmk}
\begin{em}

For the sake of shortness, we decided to limit ourselves to velocity fields and initial data with support in the ball $B_{d}(0,1)$, and we also assumed that the norm of the velocity field is less than or equal to~1. The same proof works if we consider supports contained in any ball, and any positive bound on the norm of the velocity field, or even no bound at all.

\end{em}
\end{rmk}

\begin{rmk}
\begin{em}

Probably the strange condition (\ref{hp:p4}) deserves some comment. From a technical point of view, it is just a way to rephrase the condition ``$u\in W^{1,p}(\re^{d})$ for every $p\geq 1$'' in a quantitative way that is stable by uniform convergence on $u$. There is nothing special in the choice of $p^{4}$, and several different functions of $p$ could be considered (of course they have to diverge as $p\to +\infty$, because otherwise the velocity field would be Lipschitz continuous).

In some sense, the growth of that function measures ``the lack of Lipschitz continuity'' of $u$, with slower growth corresponding to higher regularity. It could be interesting to investigate the minimal possible growth under which Theorem~\ref{thm:transport} holds true. What is sure is that not every diverging function is allowed, because we know that the Lipschitz continuity of the velocity field is not the optimal assumption that guarantees propagation of regularity for the transport equation (see for example~\cite{2011-book-BahCheDan}).

\end{em}
\end{rmk}

\begin{open}
\begin{em}

Is it possible to find a counterexample as in Theorem~\ref{thmbibl:alberti}, with the further requirement that in the last statement the function $x\to\rho(t,x)$ does not belong to $H^{s}(\Omega)$ for every $t>0$ and $s>0$, and every open set $\Omega$ contained in its support?

\end{em}
\end{open}


\subsection{Technical preliminaries}

In this subsection we limit ourselves to restating in our setting a special instance of the classical stability result for solutions to transport equations.

\begin{lemma}[see~{\cite[Theorem~II.4]{diperna-lions}}]\label{lemma:transport-wp}

Let $\left\{(u_{n},\theta_{n})\right\}\subseteq\XX$ be a sequence such that
\begin{equation}
(u_{n},\theta_{n})\to(u_{\infty},\theta_{\infty})
\quad
\mbox{in }\XX.
\nonumber
\end{equation}

Let $\left\{\rho_{n}\right\}$ and $\rho_{\infty}$ denote the corresponding solutions to problem (\ref{eqn:transport})--(\ref{eqn:transport-data}).

Then for every $T>0$ and every $p\in[1,+\infty)$ it turns pout that
\begin{equation}
\rho_{n}\to\rho_{\infty}
\qquad
\mbox{in }C^{0}([0,T],L^{p}(\re^{d})).
\nonumber
\end{equation}

\end{lemma}


\subsection{Proof of Theorem~\ref{thm:transport}}

\paragraph{\textmd{\textit{Quantitative non-pathological behavior}}}

Let $\CC$ denote the set of ``non-counterexamples'', namely the set of all pairs $(u,\theta)\in\XX$ such that the corresponding solution $\rho$ to problem (\ref{eqn:transport})--(\ref{eqn:transport-data}) satisfies
\begin{equation}
\exists t>0\quad\exists s>0
\qquad
\|\rho(t,x)\|_{H^{s}(\re^{d})}<+\infty.
\label{defn:transport-CC}
\end{equation}

We have to show that $\CC$ is a countable union of closed sets with empty interior. To this end, we state in a quantitative way the definition of $\CC$. For every positive integer $k$, we consider the set $\CC_{k}$ of all pairs $(u,\theta)\in\XX$ such that the corresponding solution $\rho$ to (\ref{eqn:wave})--(\ref{eqn:data}) satisfies
\begin{equation}
\exists t\in\left[1/k,k\right]
\qquad\quad
\|\rho(t,x)\|_{H^{1/k}(\re^{d})}\leq k.
\nonumber
\end{equation}

In words, now $t$ is confined in a compact set away from~0, we have fixed $s=1/k$, and we have prescribed a bound on the norm in $H^{s}(\re^{d})$ of the function $x\to\rho(t,x)$. 

The proof is complete if we show that $\CC$ is the union of all $\CC_{k}$'s, and that $\CC_{k}$ is a closed set with empty interior for every positive integer $k$.

\paragraph{\textmd{\textit{The set $\CC$ is the union of all $\CC_{k}$'s}}}

Let $(u,\theta)\in\CC$, so that the corresponding solution to problem (\ref{eqn:transport})--(\ref{eqn:transport-data}) satisfies (\ref{defn:transport-CC}). Then it turns out that $(u,\theta)\in\CC_{k}$ provided that we choose $k$ such that
\begin{equation}
\frac{1}{k}\leq t\leq k,
\qquad
\frac{1}{k}\leq s,
\qquad
\|\rho(t,x)\|_{H^{1/k}(\re^{d})}\leq k.
\nonumber
\end{equation}

We just need to observe that, if the function $x\to\rho(t,x)$ belongs to the space $H^{s}(\re^{d})$, then it belongs as well to the space $H^{\sigma}(\re^{d})$ for every $\sigma\in(0,s]$, and
\begin{equation}
\|\rho(t,x)\|_{H^{\sigma}(\re^{d})}\leq\|\rho(t,x)\|_{H^{s}(\re^{d})}
\qquad
\forall\sigma\in(0,s]. 
\nonumber
\end{equation}

\paragraph{\textmd{\textit{The set $\CC_{k}$ is closed}}}

Let $k$ be a \emph{fixed} positive integer. Let $\{(u_{n},\theta_{n})\}\subseteq\CC_{k}$ be any sequence, and let us assume that $(u_{n},\theta_{n})\to(u_{\infty},\theta_{\infty})$ in $\XX$, namely that $u_{n}\to u_{\infty}$ uniformly in $[0,+\infty)\times\re^{d}$, and $\theta_{n}\to\theta_{\infty}$ uniformly in $\re^{d}$ with all its derivatives. We claim that $(u_{\infty},\theta_{\infty})\in\CC_{k}$.

From the definition of $\CC_{k}$ we know that for every $n$ there exists $t_{n}\in[1/k,k]$ such that the solution $\rho_{n}$ to the transport equation (\ref{eqn:transport})--(\ref{eqn:transport-data}) with velocity field $u:=u_{n}$ and initial datum $\theta:=\theta_{n}$ satisfies
\begin{equation}
\|\rho_{n}(t_{n},x)\|_{H^{1/k}(\re^{d})}\leq k.
\nonumber
\end{equation}

Up to subsequences (not relabeled) we can always assume that $t_{n}\to t_{\infty}\in[1/k,k]$. Let $\rho_{\infty}$ be the solution to problem (\ref{eqn:transport})--(\ref{eqn:transport-data})  with velocity field $u:=u_{\infty}$ and initial datum $\theta:=\theta_{\infty}$. From Lemma~\ref{lemma:transport-wp} we deduce that 
\begin{equation}
\rho_{n}(t_{n},x)\to\rho_{\infty}(t_{\infty},x)
\quad
\mbox{in }L^{p}(\re^{d})
\qquad
\forall p\in[1,+\infty).
\nonumber
\end{equation}

Since the norm in $H^{s}$ is lower semicontinuous with respect to $L^{p}$ convergence, this proves that $\|\rho_{\infty}(t_{\infty},x)\|_{H^{1/k}(\re^{d})}\leq k$, as required.

\paragraph{\textmd{\textit{The set $\CC_{k}$ has empty interior}}}

Let us assume that there exist an integer $k_{0}\geq 1$, a pair $(u_{0},\theta_{0})\in\XX$, and a real number $\ep_{0}>0$ such that $B_{\XX}((u_{0},\theta_{0}),\ep_{0})\subseteq\CC_{k_{0}}$. 

\subparagraph{\textmd{\textit{Regularization of the center}}}

Up to a small modification of $u_{0}$ and $\theta_{0}$, and a small reduction of the radius $\ep_{0}$, we can assume that $u_{0}$ and $\theta_{0}$ have the following further properties.
\begin{itemize}

\item  They are smooth enough in the space variables (in this case it is enough to assume $C^{1}$ regularity in space, uniform with respect to time).

\item  Their support is contained in the ball $B_{d}(0,1-\ep_{1})$ for some $\ep_{1}\in(0,1)$, namely
\begin{equation}
\|u_{0}(t,x)\|+|\theta_{0}(t,x)|=0
\qquad
\forall t\geq 0,\quad\forall x\in\re^{d}\setminus B_{d}(0,1-\ep_{1}).
\label{sat:support}
\end{equation}

\item  The velocity field does not saturate the inequalities in (\ref{hp:p4}), namely
\begin{equation}
\|D_{x} u_{0}(t,x)\|_{L^{p}(\re^{d})}\leq(1-\ep_{1})p^{4}
\qquad
\forall t\geq 0,\quad\forall p\geq 1.
\label{sat:up}
\end{equation}

\end{itemize}

As a consequence, the corresponding solution $\rho_{0}$ is of class $C^{1}$ in space, uniformly with respect to time, and its support is contained in the same ball $B_{d}(0,1-\ep_{1})$.

Such an approximation can be achieved in three steps. In a first step we replace $u_{0}$ and $\theta_{0}$ with
\begin{equation}
(1-\ep)u_{0}\left(t,\frac{x}{1-\ep}\right)
\qquad\mbox{and}\qquad
\theta_{0}\left(\frac{x}{1-\ep}\right).
\nonumber
\end{equation}

If $\ep>0$ is small enough, this operation produces elements of $\XX$ as close as we want to $(u_{0},\theta_{0})$, and with a smaller support. In a second step we perform a convolution in the space variables. This procedure regularizes the functions, without enlarging too much the support, and without altering the conditions in (V2), (V3), (V4). Finally, in the third step we multiply again the velocity field by $1-\ep$ for some small enough $\ep>0$ (possibly different from the first step) in order to fulfill condition (\ref{sat:up}).

\subparagraph{\textmd{\textit{Use of rescaled basic ingredient}}}

Let us choose a point $x_{0}\in\re^{d}$ such that 
\begin{equation}
1-\ep_{1}<\|x_{0}\|<1.
\label{defn:x0}
\end{equation}

The idea is to modify $u_{0}$ and $\theta_{0}$ by attaching, in a small neighborhood of $x_{0}$, a suitable rescaling of the vector field $u_{*}$ and the function $\theta_{*}$ provided by Theorem~\ref{thmbibl:alberti-basic}. The goal is that the modified pair still belongs to $B_{\XX}((u_{0},\theta_{0}),\ep_{0})$, but it does not belong to $\CC_{k_{0}}$, thus providing a contradiction. To this end, we set
\begin{equation}
u_{n}(t,x):=n\lambda_{n}u_{*}\left(nt,\frac{x-x_{0}}{\lambda_{n}}\right)
\qquad
\forall t\geq 0,\quad\forall x\in\re^{d}
\label{defn:un-alberti}
\end{equation}
and
\begin{equation}
\theta_{n}(t,x):=\gamma_{n}\theta_{*}\left(\frac{x-x_{0}}{\lambda_{n}}\right)
\qquad
\forall x\in\re^{d},
\nonumber
\end{equation}
where $\lambda_{n}$ and $\gamma_{n}$ are two sequences of real numbers that satisfy
\begin{equation}
n\lambda_{n}\to 0,
\qquad\qquad
\frac{\gamma_{n}}{\lambda_{n}^{\alpha}}\to 0,
\qquad\qquad
\gamma_{n}\lambda_{n}^{\alpha}\exp(\beta n)\to+\infty
\label{defn:lgn}
\end{equation}
for every positive value of the parameters $\alpha$ and $\beta$, and for which there exists a constant $\ell_{d}$, depending only on the dimension, such that
\begin{equation}
n^{2}\lambda_{n}^{d/p}\leq\ell_{d}\,p^{4}
\qquad
\forall p\geq 1.
\label{defn:lgn-p}
\end{equation}

As an example, we can take
\begin{equation}
\lambda_{n}:=\exp\left(-n^{1/2}\right)
\qquad\mbox{and}\qquad
\gamma_{n}:=\exp\left(-n^{2/3}\right).
\nonumber
\end{equation}

We observe that the solution to problem (\ref{eqn:transport})--(\ref{eqn:transport-data}) with $u:=u_{n}$ and $\theta:=\theta_{n}$ is the function defined by
\begin{equation}
\rho_{n}(t,x):=\gamma_{n}\rho_{*}\left(nt,\frac{x-x_{0}}{\lambda_{n}}\right)\qquad
\forall t\geq 0,\quad\forall x\in\re^{d},
\nonumber
\end{equation}
where $\rho_{*}$ is the same as in Theorem~\ref{thmbibl:alberti-basic}.

\subparagraph{\textmd{\textit{Final contradiction: $(u_{0}+u_{n},\theta_{0}+\theta_{n})\in B_{\XX}((u_{0},\theta_{0}),\ep_{0})$ for $n$ large enough}}}

This requires some steps.

\begin{itemize}

\item  To begin with, we observe that $\lambda_{n}\to 0$. Since the support of $u_{n}$ and $\theta_{n}$ is contained in $B_{d}(x_{0},R\lambda_{n})$, from (\ref{sat:support}) and (\ref{defn:x0}) we deduce that, when $n$ is large enough, the support of $u_{0}$ and $\theta_{0}$ is disjoint from the support of $u_{n}$ and $\theta_{n}$, and therefore $\rho_{0}+\rho_{n}$ is the solution to the transport equation (\ref{eqn:transport}) with velocity field $u_{0}+u_{n}$ and initial condition $\theta_{0}+\theta_{n}$.

\item  We show that $\theta_{n}\to 0$ in $\FF$. Indeed, for every multi-index $\alpha$ it turns out that
\begin{equation}
\partial^{\alpha}\theta_{n}(x)=\frac{\gamma_{n}}{\lambda_{n}^{|\alpha|}}\cdot\partial^{\alpha}\theta_{*}\left(\frac{x-x_{0}}{\lambda_{n}}\right)
\qquad
\forall x\in\re^{d},
\nonumber
\end{equation}
and hence from the second condition in (\ref{defn:lgn}) it follows that $\partial^{\alpha}\theta_{n}\to 0$ uniformly in $\re^{d}$. This is equivalent to saying that $\theta_{n}\to 0$ in $\FF$. 

\item  We show that $u_{n}\to 0$ uniformly in $[0,+\infty)\times\re^{d}$. This will imply that $u_{0}+u_{n}\to u_{0}$ in $\VV$ as soon as we show that $u_{0}+u_{n}\in\VV$.

Indeed, from (\ref{defn:un-alberti}) and (\ref{th:bounds-ur}) it follows that
\begin{equation}
\|u_{n}(t,x)\|\leq Mn\lambda_{n}
\qquad
\forall t\geq 0,\quad\forall x\in\re^{d},
\label{est:un-unif}
\end{equation}
which implies the uniform convergence because of the first condition in (\ref{defn:lgn}).

\item  We show that $u_{0}+u_{n}\in\VV$ when $n$ is large enough.
\begin{itemize}

\item To begin with, we observe that conditions (V1) and (V4) follow from the corresponding properties of $u_{0}$ and $u_{n}$.

\item  Let us check condition (V2). Since $u_{0}$ already satisfies (V2), when the support of $u_{0}$ and $u_{n}$ is disjoint it is enough to verify that $\|u_{n}(t,x)\|\leq 1$ for all admissible values of $t$ and $x$. For $n$ large enough this follows from (\ref{est:un-unif}) and from the first condition in (\ref{defn:lgn}).

\item  Let us check condition (V3). From (\ref{defn:un-alberti}) it follows that
\begin{equation}
D_{x} u_{n}(t,x)=nD_{x} u_{*}\left(nt,\frac{x-x_{0}}{\lambda_{n}}\right)
\qquad
\forall t\geq 0,\quad\forall x\in\re^{d},
\nonumber
\end{equation}
and therefore from (\ref{th:bounds-ur}) we obtain that
\begin{equation}
\|D_{x} u_{n}(t,x)\|_{L^{p}(\re^{d})}\leq Mn\left(\omega_{d}R^{d}\lambda_{n}^{d}\right)^{1/p},
\nonumber
\end{equation}
where $\omega_{d}$ denotes the Lebesgue measure of the unit ball in $\re^{d}$. Keeping (\ref{defn:lgn-p}) into account, we deduce that
\begin{equation}
\|D_{x} u_{n}(t,x)\|_{L^{p}(\re^{d})}\leq\frac{M\ell_{d}(\omega_{d}R^{d})^{1/p}}{n}p^{4}
\qquad
\forall p\geq 1,
\nonumber
\end{equation}
and therefore from (\ref{sat:up}) we conclude that
\begin{eqnarray}
\|D_{x}(u_{0}+u_{n})(t,x)\|_{L^{p}(\re^{d})} & \leq & \|D_{x} u_{0}(t,x)\|_{L^{p}(\re^{d})}+\|D_{x} u_{n}(t,x)\|_{L^{p}(\re^{d})}  \nonumber  \\[1ex]
& \leq & \left(1-\ep_{1}+\frac{M\ell_{d}(\omega_{d}R^{d})^{1/p}}{n}\right)p^{4}.
\nonumber
\end{eqnarray}

When $n$ is large enough, the last term is less than or equal to $p^{4}$, uniformly with respect to $p\geq 1$, as required.

\end{itemize}

\end{itemize}

\subparagraph{\textmd{\textit{Final contradiction: $(u_{0}+u_{n},\theta_{0}+\theta_{n})\not\in C_{k_{0}}$ for $n$ large enough}}}

From the scaling properties of the norm in homogeneous space $\dot{H}^{s}$ we know that
\begin{equation}
\|\rho_{n}(t,x)\|_{\dot{H}^{s}(\re^{d})}=\gamma_{n}\lambda_{n}^{d/2}\lambda_{n}^{-s}\|\rho_{*}(nt,x)\|_{\dot{H}^{s}(\re^{d})}
\qquad
\forall t\geq 0.
\nonumber
\end{equation}

On the other hand, from (\ref{th:rho-growth}) we know that
\begin{equation}
\|\rho_{*}(nt,x)\|_{\dot{H}^{s}(\re^{d})}\geq C_{s}\exp(csnt)
\qquad
\forall t\geq 0,\quad\forall s\in(0,2).
\nonumber
\end{equation}

Now we set $s=1/k_{0}$. Combining these two inequalities, and keeping into account that $t\geq 1/k_{0}$ and $\lambda_{n}\leq 1$, we deduce that
\begin{equation}
\|\rho_{n}(t,x)\|_{\dot{H}^{1/k_{0}}(\re^{d})}\geq\gamma_{n}\lambda_{n}^{d/2}C_{1/k_{0}}\exp\left(\frac{c}{k_{0}^{2}}n\right)
\qquad
\forall t\geq\frac{1}{k_{0}}.
\nonumber
\end{equation}

Now we observe that
\begin{equation}
\|(\rho_{0}+\rho_{n})(t,x)\|_{\dot{H}^{1/k_{0}}(\re^{d})}\geq\|\rho_{n}(t,x)\||_{\dot{H}^{1/k_{0}}(\re^{d})}-\|\rho_{0}(t,x)\||_{\dot{H}^{1/k_{0}}(\re^{d})},
\nonumber
\end{equation}
and we recall that $\rho_{0}$ is smooth in the space variables, uniformly in time, and therefore its norm in the space $\dot{H}^{1/k_{0}}$ is uniformly bounded from above by a constant $\Gamma_{0}$ when $t\in[1/k_{0},k_{0}]$. It follows that
\begin{equation}
\|(\rho_{0}+\rho_{n})(t,x)\|_{\dot{H}^{1/k_{0}}(\re^{d})}\geq C_{1/k_{0}}\gamma_{n}\lambda_{n}^{d/2}\exp\left(\frac{c}{k_{0}^{2}}n\right)-\Gamma_{0}
\qquad
\forall t\in\left[\frac{1}{k_{0}},k_{0}\right].
\nonumber
\end{equation}

Due to the last condition in (\ref{defn:lgn}), the right-hand side tends to $+\infty$ as $n\to +\infty$, and therefore the left-hand side is larger than $k_{0}$, uniformly in $t\in[1/k_{0},k_{0}]$, when $n$ is large enough. This completes the proof.

\subsubsection*{\centering Acknowledgments}

The emphasis on qualitative vs quantitative statements (in particular within an approach based on Baire cathegory theorem) was inspired by reading the scientific work and the blog by Terence Tao. For this reason, we are grateful to him.

Both authors are members of the \selectlanguage{italian} ``Gruppo Nazionale per l'Analisi Matematica, la Probabilit\`{a} e le loro Applicazioni'' (GNAMPA) of the ``Istituto Nazionale di Alta Matematica'' (INdAM). 

\selectlanguage{english}



\label{NumeroPagine}

\end{document}